\documentclass[10pt]{article} 
\usepackage{amsmath,amssymb,amsfonts,amsthm,amscd,graphicx,psfrag,epsfig, hyperref}
\usepackage{amsmath,amssymb,amsfonts,amsthm,amscd,graphicx,psfrag,epsfig}
\usepackage{color}
\definecolor{blue}{rgb}{0,0,0.7}
\definecolor{red}{rgb}{0.75, 0, 0}
\usepackage{stmaryrd}
\usepackage[titletoc,toc]{appendix}
\input epsf

\usepackage{float}
\usepackage{marginnote}
\usepackage{amsmath, amscd, amssymb,amsfonts,amsthm,amscd,graphicx,psfrag,epsfig}
\usepackage{youngtab}
\usepackage{graphicx}
\usepackage[all]{xy}
\usepackage{array}
\usepackage{amsthm}
\usepackage{stmaryrd}

\usepackage{amsmath,amssymb,amsfonts,amsthm,amscd,graphicx,psfrag,epsfig}
\usepackage{color}
\definecolor{blue}{rgb}{0,0,0.7}
\definecolor{red}{rgb}{0.75, 0, 0}
\usepackage{stmaryrd}
\usepackage[titletoc,toc]{appendix}
\input epsf
\usepackage{hyperref}
\hypersetup{colorlinks, linkcolor=blue, citecolor=cyan}
\usepackage[all]{xy}           
\usepackage{marginnote}  
\usepackage{float}
\usepackage{soul}

\hypersetup{colorlinks, linkcolor=blue}
\usepackage[all]{xy}           
\usepackage{marginnote}  
\usepackage{float}
\usepackage{amsmath,amssymb,amsthm,color}
\usepackage{pgf,tikz,pgfplots}
\usepackage{mathrsfs}
\usepackage{caption}
\usepackage{subcaption}
\usetikzlibrary{arrows}
\pgfplotsset{compat=1.15}
\usetikzlibrary{matrix}

 \usepackage{graphicx}
 \include{graphics}
\usepackage{amsmath,amssymb,amsfonts,amsthm,amscd,graphicx,psfrag,epsfig,mathabx}
\usepackage{color}
\definecolor{blue}{rgb}{0,0,0.7}
\definecolor{red}{rgb}{0.75, 0, 0}
\usepackage[titletoc,toc]{appendix}

\usetikzlibrary{arrows,chains,shapes,matrix,positioning,scopes} 
\usetikzlibrary{decorations.markings}

\tikzstyle arrowstyle=[scale=1.1]
\tikzstyle directed=[postaction={decorate,decoration={markings,
    mark=at position 1 with {\arrow[arrowstyle]{latex}}}}]
\tikzstyle reverse directed=[postaction={decorate,decoration={markings,
    mark=at position .45 with {\arrowreversed[arrowstyle]{latex};}}}]    
\tikzstyle arrowstyle=[scale=1.1]
\tikzstyle directed=[postaction={decorate,decoration={markings,
    mark=at position 1 with {\arrow[arrowstyle]{latex}}}}]
\tikzstyle reverse directed=[postaction={decorate,decoration={markings,
    mark=at position .45 with {\arrowreversed[arrowstyle]{latex};}}}]
    
 \usetikzlibrary{decorations.pathreplacing,decorations.markings,decorations.pathreplacing,arrows,shapes}   
\tikzset{
  -z>/.style={
    decoration={
      show path construction,
      lineto code={
        \path (\tikzinputsegmentfirst) -- (\tikzinputsegmentlast) coordinate[pos=#1] (mid);
        \draw (\tikzinputsegmentfirst) -- (mid);
        \draw[double distance=1.5pt, arrows = {- Latex[length=0pt 2 0]}] (mid) -- (\tikzinputsegmentlast);      }
    },decorate
  }, -z>/.default=.5,
  z->/.style={
    decoration={
      show path construction,
      lineto code={
          \path (\tikzinputsegmentfirst) -- (\tikzinputsegmentlast) coordinate[pos=#1] (mid);
                \draw[double distance=1.5pt] (\tikzinputsegmentfirst) -- (mid);
                \draw[decoration={markings, mark=at position 1 with {\arrow[scale=1.2]{latex}}},postaction={decorate}] (mid) -- (\tikzinputsegmentlast);
      }
    },decorate
  }, z->/.default=.5,
  --z>/.style={
    decoration={
      show path construction,
      lineto code={
        \path (\tikzinputsegmentfirst) -- (\tikzinputsegmentlast) coordinate[pos=#1] (mid);
        \draw [dashed](\tikzinputsegmentfirst) -- (mid);
        \draw [double distance=1.5pt, dashed, arrows = {- Latex[length=0pt 2 0]}] (mid) -- (\tikzinputsegmentlast);      }
    },decorate
  }, --z>/.default=.5,
    z-->/.style={
    decoration={
      show path construction,
      lineto code={
          \path (\tikzinputsegmentfirst) -- (\tikzinputsegmentlast) coordinate[pos=#1] (mid);
                \draw[double distance=1.5pt, dashed] (\tikzinputsegmentfirst) -- (mid);
                \draw[dashed, decoration={markings, mark=at position 1 with {\arrow[scale=1.2]{latex}}},postaction={decorate}] (mid) -- (\tikzinputsegmentlast);
      }
    },decorate
  }, z-->/.default=.5,
   x->/.style={
    decoration={
      show path construction,
      lineto code={
          \path (\tikzinputsegmentfirst) -- (\tikzinputsegmentlast) coordinate[pos=#1] (mid);
                \draw[double distance=2pt] (\tikzinputsegmentfirst) -- (mid);
                \draw[decoration={markings, mark=at position 1 with {\arrow[scale=1.2]{latex}}},postaction={decorate}] (\tikzinputsegmentfirst) -- (\tikzinputsegmentlast);
      }
    },decorate
  }, x->/.default=.5,
  -x>/.style={
    decoration={
      show path construction,
      lineto code={
        \path (\tikzinputsegmentfirst) -- (\tikzinputsegmentlast) coordinate[pos=#1] (mid);
        \draw[double distance=2pt, shorten >=5pt] (mid) -- (\tikzinputsegmentlast); 
        \draw[decoration={markings, mark=at position 1 with {\arrow[scale=1.5]{latex}}},postaction={decorate}] (\tikzinputsegmentfirst) -- (\tikzinputsegmentlast);}
    },decorate
  }, -x>/.default=.5,
   x-->/.style={
    decoration={
      show path construction,
      lineto code={
          \path (\tikzinputsegmentfirst) -- (\tikzinputsegmentlast) coordinate[pos=#1] (mid);
                \draw[double distance=2pt, dashed] (\tikzinputsegmentfirst) -- (mid);
                \draw[dashed, decoration={markings, mark=at position 1 with {\arrow[scale=1.2]{latex}}},postaction={decorate}] (\tikzinputsegmentfirst) -- (\tikzinputsegmentlast);
      }
    },decorate
  }, x-->/.default=.5,
  --x>/.style={
    decoration={
      show path construction,
      lineto code={
        \path (\tikzinputsegmentfirst) -- (\tikzinputsegmentlast) coordinate[pos=#1] (mid);
        \draw[double distance=2pt, dashed, shorten >=5pt] (mid) -- (\tikzinputsegmentlast); 
        \draw[dashed, decoration={markings, mark=at position 1 with {\arrow[scale=1.4]{latex}}},postaction={decorate}] (\tikzinputsegmentfirst) -- (\tikzinputsegmentlast);}
    },decorate
  }, --x>/.default=.5,
}

\tikzset{
    partial ellipse/.style args={#1:#2:#3}{
        insert path={+ (#1:#3) arc (#1:#2:#3)}
    }
}

\newtheorem{theorem}{Theorem}[section]

\newtheorem{theorem-definition}[theorem]{Theorem-Definition}
\newtheorem{theorem-construction}[theorem]{Theorem-Construction}
\newtheorem{lemma-definition}[theorem]{Lemma--Definition}
\newtheorem{lemma-construction}[theorem]{Lemma--Construction}
\newtheorem{lemma}[theorem]{Lemma}
\newtheorem{proposition}[theorem]{Proposition}
\newtheorem{corollary}[theorem]{Corollary}
\newtheorem{conjecture}[theorem]{Conjecture}
\newtheorem{definition}[theorem]{Definition}
\newtheorem{example}[theorem]{Example}

\newcommand{\old}[1]{}

\newcommand{\rH}{{\rm H}}

\newcommand{\Z}{{\mathbb Z}}
\newcommand{\R}{{\mathbb R}}
\newcommand{\C}{{\mathbb C}}

\newcommand{\A}{{\rm A}}
\newcommand{\F}{{\rm F}}
\newcommand{\B}{{\rm B}}
\newcommand{\G}{{\rm G}}
\newcommand{\U}{{\rm U}}

\newcommand{\Q}{{\Bbb  Q}}

\newcommand{\bS}{{{\Bbb S}}}

\newcommand{\lms}{\longmapsto}
\newcommand{\lra}{\longrightarrow}
\newcommand{\hra}{\hookrightarrow}

\newcommand{\be}{\begin{equation}}
\newcommand{\ee}{\end{equation}}
\newcommand{\bt}{\begin{theorem}}
\newcommand{\et}{\end{theorem}}
\newcommand{\bd}{\begin{definition}}
\newcommand{\ed}{\end{definition}}
\newcommand{\bp}{\begin{proposition}}
\newcommand{\ep}{\end{proposition}}

\newcommand{\bl}{\begin{lemma}}
\newcommand{\el}{\end{lemma}}
\newcommand{\bc}{\begin{corollary}}
\newcommand{\ec}{\end{corollary}}
\newcommand{\bcon}{\begin{conjecture}}
\newcommand{\econ}{\end{conjecture}}
\newcommand{\la}{\label}

\begin{document}

\date{}

\title {Cluster construction of the second  motivic Chern class}
\author{Alexander B. Goncharov, Oleksii Kislinskyi}

\maketitle
\tableofcontents

\begin{abstract}

Let $\G$ be a split, simple, simply connected, algebraic group over $\Q$. The degree 4, weight 2 motivic cohomology group of the classifying space $\B\G_\bullet$ of $\G$ is identified with 
$\Z$. We construct cocycles representing the generator, known as the second universal motivic Chern class.\\

  If $\G = {\rm SL(m)}$, there is a canonical cocycle, defined by the first author (1993). For any group G, we define a collection of cocycles parametrised by cluster coordinate systems on the space of 
  $\G-$orbits on the cube of the principal affine space $\G/\U$. Cocycles for different clusters are related by explicit coboundaries, constructed using cluster transformations relating the clusters.\\
  
  The cocycle has three components. The construction of the last one is canonical and elementary; it does not use clusters, and provides the   motivic generator of  $\rH^3(\G(\C), \Z(2))$. However to lift it to the whole cocycle we need cluster coordinates:  construction of the first two components uses crucially the cluster structure of the moduli spaces ${\cal A}(\G,\bS)$ related to the moduli space of $\G-$local systems on $\bS$. 
  In retrospect, it partially explains why   cluster coordinates on the space ${\cal A}(\G,\bS)$ should exist.\\
  
  The construction has numerous applications, including  explicit constructions of the universal extension of the group $\G$ by $K_2$, the line bundle on ${\rm Bun}(\G)$
   generating its Picard group, Kac-Moody groups, etc. Another application is an explicit combinatorial construction of the second motivic Chern class of a $\G$-bundle. It is a motivic analog of the work of Gabrielov-Gelfand-Losik (1974), for any $\G$.\\
  
  We show that the cluster construction of the measurable group 3-cocycle for the group $\G(\C)$, provided by the component 
  ${\rm C}^{(1)}$ of our motivic cocycle,  
 gives rise to  its quantum deformation. 
  We conjecture that there are   quantum deformations of all measurable  cohomology classes of $\G(\C)$. 
   \end{abstract}

 \section{Summary}   Let $\G$ be a  split, simple,  simply connected, algebraic group over $\Bbb Q$. Let   $\B\G$ be the classifying space of $\G$, and    $\underline \Z_{\cal M}(2)$     the  weight two motivic cohomology complex.
 It is well known  that
   \be \la{BD1*}
 \rH^4(\B\G, \underline \Z_{\cal M}(2))  \stackrel{\tau}{=} \rH^3(\G, \underline \Z_{\cal M}(2))  \stackrel{}{=} \rH^1(\G, \underline K_2)  \stackrel{}{=}  \rH_{\rm Betti}^3(\G, \Z(2))   \stackrel{}{=}\Z.
\ee 
The isomorphism $\tau$  is the transgression in the motivic cohomology for the universal $\G-$bundle over $\B\G$. The second follows using the  Gersten resolution. 
The third isomorphism was established   by  Brylinsky-Deligne    \cite{BD}, and the last one is well known.  \\

Given a regular variety $X$, we use the complex $\Z^\bullet_{\cal M}(X;2)$,  defined via the Gersten resolution of the Bloch complex,  
\cite{G91},  see also (\ref{30}), and calculating the rational motivic cohomology of $X$:\footnote{By the very definition,  $\rH^i( \Z^\bullet_{\cal M}(X; 2))=0$ for $i\leq 0$, while  $\rH^i( X, \underline \Z^\bullet_{\cal M}( 2))=0$ is the Beilinson-Soule conjecture. Note  that $\rH^1( \Z^\bullet_{\cal M}({\rm Spec}(\F); 2))$ does not contain the subgroup of $K_3^{\rm ind}(\F) $ given, by \cite{S},  by the $\Z/2\Z$ extension of ${\rm Tor}(\F^\times, \F^\times)$. }\be
\begin{split}
&\rH^i( \Z^\bullet_{\cal M}(X; 2))  \stackrel{}{=} \rH^i(X, \underline \Z^\bullet_{\cal M}(2)), ~~~~i \geq 2.\\
&\rH^1( \Z^\bullet_{\cal M}(X; 2))\otimes \Q  \stackrel{}{=} \rH^1(X, \underline \Z^\bullet_{\cal M}(2))\otimes \Q.
\end{split}
\ee
Its definition extends immediately to the case when $X$ is a regular simplicial scheme. We use Milnor's simplicial model $\B\G_\bullet$ 
of   $\B\G$, see paragraph 10. Then (\ref{BD1*})    looks as follows:
   \be \la{BD1}
    \rH^4( \Z^\bullet_{\cal M}(\B\G_\bullet; 2))  \stackrel{\tau}{\lra} 
 \rH^3(\Z^\bullet_{\cal M}(\G; 2))  \stackrel{}{=} \rH^1(\G, \underline K_2)  \stackrel{}{=}  \rH_{\rm Betti}^3(\G, \Z(2))   \stackrel{}{=}\Z.
\ee 
Here the map $\tau$ is surjective, and his kernel is a torsion subgroup. \\

 We construct   cocycles ${\rm C}^\bullet$ representing  the  
  second universal motivic Chern class,    i.e. an element
   \be \la{C2a1}
 {\rm c}_2 \in \rH^4( \Z_{\cal M}(\B\G_\bullet;2)). 
 \ee 
 such that $\tau(c_2)=1 \in \Z$ under isomorphisms (\ref{BD1}). 
If $\G = {\rm SL}_m$, there is a canonical  cocycle,  defined  in \cite{G93}.   
Given a  representation $V$ of   $\G$, it  induces a cocycle for   $\B\G$.  
Yet this way we can get  only  multiples of ${\rm c}_2$,  e.g.  $60\cdot {\rm c}_2$ for  ${\rm E_8}$. \\

 For any group $\G$, we define a collection of  cocycles ${\rm C}^\bullet$, parametrised  by   cluster coordinate systems on the 
   space of   $\G-$orbits on the cube   of the principal affine space $\G/\U$. Cocycles 
 for different clusters are related by explicit coboundaries,  constructed using    cluster transformations relating the   clusters.

A cocycle ${\rm C}^\bullet$ has three components: ${\rm C}^{(1)}, {\rm C}^{(2)}, {\rm C}^{(3)}$.  
The construction of the  component ${\rm C}^{(3)}$  is canonical and elementary;  it does not use clusters, and provides 
  a  canonical cocycle for the generator of $\rH_{\rm Betti}^3(\G, \Z(2))$.      
 However   to lift   ${\rm C}^{(3)}$ to a cocycle ${\rm C}^{\bullet}$   we need   cluster coordinates:   
 the construction of the first two components uses crucially   the    cluster  structure of the   moduli spaces 
  $ {\cal A}_{\G, \bS} $,  related to the  $\G-$character varieties for   decorated surfaces  $\bS$ \cite{GS19}.   

In retrospect, it partially explains    why the cluster coordinates on the   space  ${\cal A}_{\G, \bS}$  should exist. 
 \\
    
This construction has numerous  applications, including an  explicit construction   of the universal extension of the group $\G$ by $K_2$, the 
determinant line bundle on ${\rm Bun}_\G$,  Kac-Moody groups, etc.

Another application is  an explicit combinatorial construction of the second motivic Chern class of a $\G-$bundle. It   
 is a motivic analog of the work of Gabrielov-Gelfand-Losik \cite{GGL}, for any  $\G$. \\
 
The cluster construction of the second motivic Chern class  also provides its quantum deformation.  In Section \ref{S9}
 we explain a part of the story: a quantum deformation of the  third cohomology class of the measurable cohomology of the Lie group 
 $\G_\C:=\G(\C)$: 
 $$
 \beta_3 \in {\rm H}^3_{\rm meas}(\G_\C, \R).
 $$

   \section{Introduction and   main results}

\paragraph{1. The group $\rH^3(\G, \Z)$.} In this paper  $\G$ is a  split, simple,  simply connected algebraic group over $\Bbb Q$.   
 Its Lie algebra  $\mathfrak{g}$  is a Lie algebra over $\Q$. 
The de Rham cohomology group   $\rH_{\rm DR}^3(\G;\Q)$ is   
 identified with  invariant bilinear symmetric forms $\langle \ast, \ast \rangle$ on $\mathfrak{g}$:
\be \la{ISO}
\rH_{\rm DR}^3(\G; \Q) = S^2(\mathfrak{g}^*)^\G \stackrel{\sim}{=} \Q.
\ee
Namely,    a form  $\langle \ast, \ast \rangle \in S^2(\mathfrak{g}^*)^\G$  gives rise to the  ${\rm Ad}_\G-$invariant  $3-$form on $\mathfrak{g}$:
  \be
\begin{split}
&\varphi_{\langle \ast, \ast \rangle} \in \Lambda^3(\mathfrak{g}^*)^\G, ~~~~~~
 \varphi_{\langle \ast, \ast \rangle}(A,B,C):=  \langle A, [B,C] \rangle.\\
\end{split}
 \ee
It determines a closed  biinvariant differential $3-$form   on $\G$,   providing  isomorphism (\ref{ISO}).  For example, for $\G={\rm SL}_m$ we get   rational multiples of   
 the  form  ${\rm Tr}(g^{-1}dg)^3$.  
  
 Let  $\mathfrak{h}$ be the Lie algebra of the Cartan group $\rm H$ of $\G$, and $W$   the Weyl group of $\G$. Then 
\be
S^2(\mathfrak{g}^*)^\G = S^2(\mathfrak{h}^*)^W.
\ee
It is known that the  canonical generator of   $\rH_{\rm DR}^3(\G; \Z)$ is provided by the Killing form normalized so that its value on the shortest   coroot is equal to $1$. 
We call it the DeRham generator. 

  Denote by $\rH^*_\B$ the singular (Betti) cohomology of a topological space.  
  The integration provides an isomorphism between the DeRham and Betti cohomology, and  identifies the   generators: 
\be \la{int}
\int: \rH_{\rm DR}^3(\G; \Z)  \stackrel{\sim}{\lra} \rH_\B^3(\G(\C); \Z(2)), \ \ \ \  \Z(2) := (2\pi i)^2 \Z.
\ee

  Denote by $\B\G$ the classifying space for the algebraic group $\G$.  
  It is well known that  
  \be
  \rH^4(\B\G,   \Z(2)) = \rH^3(\G, \Z(2)).
  \ee
To introduce the motivic upgrade of this isomorphism, 
we recall   the weight two motivic complex. 

\paragraph{2. The  $K_2-$cohomology.}  Given a field $\F$,  the Milnor $K_2-$group of $\F$  is the abelian group  given by  the quotient of the wedge square $\Lambda^2\F^\times$ of the multiplicative group $\F^\times $  by the subgroup generated by the Steinberg relations $(1-x)\wedge x$, where $x \in \F^\times -\{1\}$:  
\be
K_2(\F):= \Lambda^2\F^\times /\langle (1-x)\wedge x\rangle.
\ee
Let $X$ be a regular algebraic variety over a field $k$, with the field of   functions $k(X)$. Denote by 
 $X_d$ the set of   irreducible subvarieties of codimenion $d$ on $X$. Then there is a complex  of abelian groups: 
 \be \la{RESSS}
\underline {K^\bullet_2}:=    K_2( k(X)) \stackrel{\rm res}{\lra} \bigoplus_{D\in X_1} k(D)^\times \stackrel{\rm val}{\lra} \bigoplus_{X_2}\Z.
 \ee
We place it    in the degrees $[0,2]$. The right map is the valuation map. The left map is the tame symbol:
\be \la{RESS}
{\rm res}: f\wedge g \lms \sum_{D \in X_1}(-1)^{{\rm val}_D(f){\rm val}_D(g)}\frac{f^{{\rm val}_D(g)}}{g^{{\rm val}_D(f)}}{|D}.
\ee

We denote    its cohomology by 
$ \rH^\ast(X,   \underline {K_2})$. 
 
\paragraph{3. The Hodge regulator map.} For a   regular complex algebraic variety $X$, the group $\rH^1(X, \underline K_2)$ provides 
 {\it some} elements of $\rH^3(X(\C); \Z(2))$ of the Hodge type $(2,2)$, defined  as currents of algebraic-geometric origin as follows. 
Given a divisor $D\subset X$ and a rational function $f$ on $D$, there is a $3-$current $\psi_{D, f}$ on $X(\C)$ 
whose value on a  smooth differential form $\omega$   is   
 \be
\psi_{D, f}(\omega):=   2\pi i \cdot\int_{D(\C)} d\log (f) \wedge \omega. 
 \ee
Its differential  is the  $\delta-$current, given by the integration along the codimension two cycle on $X$ provided by the divisor ${\rm div}(f)$ of   $f$: 
\be
d\psi_{D, f} = (2\pi i)^2\delta_{{\rm div}(f)}.
\ee
The   cycles in the complex calculating $\rH^1(X, \underline K_2)$ are given by   linear combinations 
 \be \la{CYC}
 \sum_i  (D_i, f_i), ~~~~\sum_i   {\rm div}(f_i)=0.
 \ee
Here $D_i$ is an irreducible divisor  in $X$, and $f_i$  a rational function on $D_i$. The  cocycle condition   implies  that 
 the $3-$current $\sum_i\psi_{D_i, f_i}$ is closed, defining an element of $\rH^3(X(\C), \Z(2))$ of the Hodge type $(2,2)$. Denote the subgroup of such classes as $\rH_{2,2}^3(X(\C), \Z(2))$.
 So we get the {\it Hodge regulator map}
 \be
  {\rm reg}_{\cal H}: \rH^1(X;  \underline  K_2)  \stackrel{ }{\lra}  \rH_{2,2}^3(X(\C); \Z(2)).
    \ee
 Beilinson's generalized Hodge conjecture \cite{Be} predicts  that it is  an isomorphism modulo torsion. This generalises  the   Hodge conjecture isomorphism for the codimension two cycles:
  \be
  \rH^2(X;  \underline  K_2)\otimes \Q = {\rm CH}^2(X) \otimes \Q  \stackrel{\sim}{\lra} \rH_{2,2}^4(X(\C); \Q(2)).
    \ee
Our next goal is  an explicit description of the group $\rH^3(\G(\C), \Z(2))$ via the Hodge regulator map. 

 \paragraph{4. The generator of the group $\rH^1(\G, \underline{K}_2) = \Z$.}  Denote by ${\rm I}$ the set of   vertices of the Dynkin diagram for the group $\G$. Let ${\rm C}_{ij}$,  $i,j\in {\rm I}$,  be the Cartan matrix. Recall the  Bruhat decomposition   
 of   $\G$:
\be
G = \coprod_{w\in W}  {\cal B}_w, ~~~~ {\cal B}_w:= \U  \rH\overline w \U. 
\ee
Here $\overline w$ is the canonical lift of a Weyl group element $w$ to $\G$. Therefore, given a Weyl group element $w\in W$  and a character $\chi$ of the Cartan group $\rH$, we get    a regular function $\chi_w$ on the Bruhat cell ${\cal B}_w$:
\be \la{xw}
\chi_w \in {\cal O}^\times( {\cal B}_w), ~~~~~~ \chi_w(u_1 h \overline w  u_2):= \chi(h).
\ee
The dominant weight $\Delta_k$ gives rise to a regular  function on the Bruhat cell $ {\cal B}_w$, denoted by $\Delta_{k, w}$.

 Recall the longest element $w_0$ of $W$. The Bruhat divisor $ {\cal B}_{s_kw_0}$ is determined by the equation $\Delta_{k, w_0}=0$.    Let us introduce the following rational function 
 on this divisor. Denote by $i_k: {\cal B}_{s_kw_0} \subset \G$ the natural embedding. Set
\be
F_k:= i_k^*\Bigl(\Delta^{-1}_{k, s_k \omega_0}\prod_{i \in {\rm I} \setminus \{k\}}    (\Delta_{i, w_0})^{\frac{C_{ki}}{2}}\Bigr)^{d_k}.
\ee 
Here the integers $\{d_i\}$ are the symmetrizers: $d_i{\rm C}_{ij} = d_j{\rm C}_{ji}$. 
Let us consider the following formal sum of the pairs (a Bruhat divisor, a rational function on it):
\be
{\rm C}^{(3)}= \bigoplus_{k \in {\rm I}} \Bigl( {\cal B}_{s_kw_0}, F_k \Bigr).
\ee
\bt \la{TH1}
The element ${\rm C}^{(3)}$ is a 1-cocycle in the complex $\underline{K}^\bullet_2\otimes \Z[\frac{1}{2}]$. 
Its cohomology class $[{\rm C}^{(3)}]$     generates   the group $\rH^1(\G, \underline{K}_2) = \Z$. Its Hodge realization ${\rm reg}_{\cal H}[{\rm C}^{(3)}]$ generates 
the group $\rH^3(\G(\C), \Z(2))$.
\et

 \paragraph{5. An example:  $\rH^3({\rm SL_2}(\C))$.}  
 There are three    ways to describe this  group:

 \begin{enumerate}
 
\item   {\it Betti}.  One has $\rH^3_\B(SL_2(\C); \Z)  =\Z$ since $SU(2) = S^3$ is a  retract of $SL_2(\C)$.  
 
\item  {\it De Rham}.  The generator of  $H_{\rm DR}^3({\rm SL_2}; \Z)$ is given by the form ${\rm Tr}(g^{-1}dg)^3$ on ${\rm SL_2}$.  
The coefficient  $\Z(2)$ in the comparison isomorphism (\ref{int}) reflects  the volume formula ${\rm vol}(S^3) = 2 \pi^2$.  

\item  {\it Motivic}.    A line $L$ in a 2-dimensional vector space $V$  provides  a divisor $\B_L$ with a function $f$: 
$$
\B_L  := \{g \in SL_2 ~|~ gL =L\}, ~~~~~~gl = f(g)l, ~~~\forall g \in \B_L, ~~l \in L.
$$  The 3-current $\psi_{\B_L, f}$  generates $\rH^3_\B(SL_2(\C); \Z(2))$. 
 
\end{enumerate}

Theorem \ref{TH1} is proved in Section \ref{S2}. 
The group $\rH^1(\G, \underline{K}_2)$ was described   by Brylinsky-Deligne \cite{BD}.  
Theorem \ref{TH1} provides a   specific cocycle for the   generator of $ \rH^1(\G, \underline{K}_2) $.    
Such a cocycle, of course,    is   not unique. 
Our   cocycle is tied up with  the cluster structure of the space ${\cal A}_{\G, \bS}$.   Let us elaborate on this.

\paragraph{6. The key feature of the cocycle ${\rm C}^{(3)}$.} We identify  $ \rH^1(\G, \underline{K}_2) $ with the $\G-$invariants  
$ \rH^1(\G\times \G, \underline{K}_2)^\G $, for the left diagonal action of $\G$.  There are three  projections 
\be \la{PP}
p_{ij}:\G^3 \lra \G^2, ~~~~1 \leq i < j \leq 3,~~~~p_{ij}(g_1, g_2, g_3) := (g_i, g_j). 
\ee
We claim that 
\be
p_{12}^*[{\rm C}^{(3)}]  + p_{23}^*[{\rm C}^{(3)}]  - p_{13}^*[{\rm C}^{(3)}] =0.
 \ee
 Our goal is to prove this on the level of complexes,  constructing explicitly a $\G-$invariant  element of $K_2(\G^3)$ whose residue is the cocycle representing the cohomology class on the 
 left. This boils down to a construction of a certain 
 $\G-$invariant element ${\rm C}^{(2)}$ in 
 $\Q(\G^3)^* \wedge  \Q(\G^3)^*$. 
 
 \paragraph{7. The element ${\rm C}^{(2)}$.} Observe that $ 
 \U\backslash \G/\U = \G\backslash (\G/\U \times \G/U)$. By the construction, the  cocycle ${\rm C}^{(3)}$ is invariant under  the right action of the group $\U \times \U$ on $\G\times \G$.  
  Note that  projections (\ref{PP}) determine   three similar canonical projections involving ${\cal A}:= \G/\U$ which are denoted, abusing notation, by
\be
p_{ij}: {\cal A}^3 \lra {\cal A}^2 ~~~~ 1 \leq i < j \leq j.
\ee

So we are looking for an  element   
 \be \la{CCA}
 \begin{split}
 &{\rm C}^{(2)}\in  \Q({\cal A}^3)^* \wedge  \Q({\cal A}^3)^*\\
 &{\rm res}({\rm C}^{(2)} )= p_{12}^*{\rm C}^{(3)}  + p_{23}^*{\rm C}^{(3)}  - p_{13}^*{\rm C}^{(3)}.\\
 \end{split}
  \ee
 Explicitely, we can write 
\be\begin{split}
&{\rm C}^{(2)}= \sum_{i,j} \widetilde \varepsilon_{ij} \cdot \A_i\wedge A_j, ~~~~\A_i \in  \Q( {\cal A}^3)^\G ~~~~ \widetilde\varepsilon_{ij} = -  \widetilde\varepsilon_{ji}\in \Z.\\
\end{split}
\ee
Here $\{A_j\}$ is a collection of   $\G-$invariant regular functions on ${\cal A}^3$. So to construct ${\rm C}^{(2)}$ we must   exibit a collection of such functions. 
This is  exactly what the cluster structure on the space ${\rm Conf}_3({\cal A}):=\G\backslash {\cal A}^3$ does:  the functions $\{A_j\}$   are the cluster coordinates, and 
$ \widetilde\varepsilon_{ij}$ is the skew-symmetrized exchange matrix.

 The element ${\rm C}^{(2)}$ is defined in the end of Section \ref{SECt4}, where we recall the construction of a cluster for the space ${\rm Conf}_3({\cal A})$.  
 Different cluster coordinate systems deliver   elements ${\rm C}^{(2)}$ which differ  by explicitly given    sums of Steinberg relations,  
and therefore define the same class in $K_2$. 

Note   that the cluster structure does more: it delivers   elements 
 where the number of functions $A_i$   equals to the dimension of  
${\rm Conf}_3({\cal A})$, and these functions are regular coordinates on this space. 

On the other hand, this partially explains why the cluster coordinates on   ${\rm Conf}_3({\cal A})$ should exist: we know  that an element (\ref{CCA}) must exist.

\paragraph{8. Remark.}  A similar problem for the deRham cocycle is much easier, and has a canonical solution:
\be
3\cdot d {\rm Tr} \Bigl(g_1g_2 dg^{-1}_2dg_1^{-1} \Bigl) = {\rm Tr}(g_1^{-1}dg_1)^3 + {\rm Tr}(g_2^{-1}dg_2)^3 - {\rm Tr}\Bigl((g_1g_2)^{-1}d(g_1g_2)\Bigl)^3.
  \ee 
  To explain the general problem, and  how the   elements ${\rm C}^{(2)}, {\rm C}^{(3)}$  fit in the motivic framework, we   
recall   two basic ingredients of the construction: the weight two motivic complex, and   Milnor's    model for  ${\rm B}\G$.  

\paragraph{9. The weight two motivic complex.}

Recall the  cross-ratio  of four points on ${\Bbb P}^1(\F)$: 
\be \la{CR1}
r(s_1, s_2, s_3, s_4) := \frac{(s_1- s_2) (s_3- s_4) }{(s_1- s_4) (s_2- s_3) }, ~~~~r(\infty, -1, 0, z) = z.
\ee
Given any five distinct points $s_1, ..., s_5$ on $ {\Bbb P}^1(\F)$, consider the element 
\be \la{AFED}
\sum_{i=1}^5 \{-r(s_i, s_{i+1}, s_{i+2},  s_{i+3})\} \in \Z[\F], ~~~~i \in \Z/5\Z.
\ee
Denote by ${\rm R}_2(\F)$ the   subgroup of $\Z[\F^*-\{1\}]$ generated by   elements (\ref{AFED}) for all  5-tuples of distinct points. 
The Bloch group $\B_2(\F)$ is  the quotient
\be
\B_2(\F):= \frac{\Z[\F^*-\{1\}]}{{\rm R}_2(\F)}.
\ee
The key point is that there  is a well defined map 
\be \la{BLOCH}
\begin{split}
 \delta: \ &\B_2(\F) \lra  \F^*\wedge \F^*.\\
&\{x\}\lms (1-x) \wedge x.\\
\end{split}
\ee
This complex, placed in the degrees $[1,2]$, is called the {\it Bloch complex}. Note that ${\rm Coker}(\delta) = K_2(\F)$. 

Let $X$ be a regular algebraic variety over a field $k$. Then there is a complex  of abelian groups placed in the degrees $[1,4]$, and called the {\it weight two motivic complex} of $X$: 
 \be \la{30}
\Z^\bullet_{\cal M}(X; 2):=  \B_2(k(X)) \stackrel{\delta}{\lra}  k(X)^*\wedge k(X)^* \stackrel{\rm res}{\lra} \bigoplus_{D\in X_1} k(D)^\times \stackrel{\rm val}{\lra} \bigoplus_{X_2}\Z.
 \ee
We use the notation
\be
 \rH^\ast(X,   \underline \Z_{\cal M}(2)):=  \rH^*(\Z^\bullet_{\cal M}(X; 2)).
 \ee
Recall that 
  \be \la{II}
  \rH^4(\B\G,   \underline \Q_{\cal M}(2)) = \rH^3(\G, \underline \Q_{\cal M}(2))  = S^2(\mathfrak{h}^*)^W = \Q.
  \ee

  \bd The  second universal motivic Chern class
 \be \la{C2}
 {\rm c}_2 \in \rH^4(\B\G, \underline \Z_{\cal M}(2)) 
 \ee
  is the integral generator  which corresponds, under   isomorphisms (\ref{II}), to the Killing  
  form on $\mathfrak{g}$ normalized so that its values on the shortest coroot is equal to $1$. 
 \ed

\paragraph{10. Milnor's simplicial model ${\rm B}\G_\bullet$ of the classifying space ${\rm B}\G$.}  Recall the simplicial realization ${\rm E}\G_\bullet$ of the space ${\rm E}\G$: \\
 
 \begin{tikzpicture}
\path  
 node (a) at (4,0) {$\cdots$}
  node (b) at (5,0) {$G^3$}
  node (c) at (8,0) {$G^2$}
  node (d) at (11,0) {$G$};

\begin{scope}[->,>=latex]

    \foreach \i in {0,...,2}{%
      \draw[->] ([yshift=(\i - 1) * 0.6 cm]b.east) -- ([yshift= (\i - 1) * 0.6 cm]c.west) ;}
    \foreach \i in {0,...,1}{%
      \draw[<-] ([yshift=(2 * \i -1) * 0.3 cm]b.east) -- ([yshift=(2 * \i -1) * 0.3 cm]c.west) ;}

    \foreach \i in {0,...,1}{%
      \draw[->] ([yshift=(2 *\i - 1) * 0.3 cm]c.east) -- ([yshift=(2 *\i - 1) * 0.3 cm]d.west) ;}
    \foreach \i in {0,...,0}{%
      \draw[<-] ([yshift=\i * 0.3 cm]c.east) -- ([yshift=\i * 0.3 cm]d.west) ;}
\end{scope}

\end{tikzpicture} \\
 In particular, there are the $n+1$ standard maps  
\be
s_{n, i}: \G^{n+1} \lra \G^n, ~(g_0, ..., g_n) \lms (g_0, ..., \widehat g_i, ..., g_n), ~~i=0, ..., n.
\ee
Then we set 
 ${\rm B}\G_\bullet:= \G\backslash {\rm E}\G_\bullet$:\\
 
  \begin{tikzpicture}
\path  
 node (a) at (4,0) {$\cdots$}
  node (b) at (5,0) {$G^2$}
  node (c) at (8,0) {$G$}
  node (d) at (11,0) {$\ast$};

\begin{scope}[->,>=latex]

    \foreach \i in {0,...,2}{%
      \draw[->] ([yshift=(\i - 1) * 0.6 cm]b.east) -- ([yshift= (\i - 1) * 0.6 cm]c.west) ;}
    \foreach \i in {0,...,1}{%
      \draw[<-] ([yshift=(2 * \i -1) * 0.3 cm]b.east) -- ([yshift=(2 * \i -1) * 0.3 cm]c.west) ;}

    \foreach \i in {0,...,1}{%
      \draw[->] ([yshift=(2 *\i - 1) * 0.3 cm]c.east) -- ([yshift=(2 *\i - 1) * 0.3 cm]d.west) ;}
    \foreach \i in {0,...,0}{%
      \draw[<-] ([yshift=\i * 0.3 cm]c.east) -- ([yshift=\i * 0.3 cm]d.west) ;}
\end{scope}

\end{tikzpicture} \\

Let   $X \lms {\cal F}^\bullet(X)$ be an assignment to an algebraic variety $X$ a  complex of abelian groups ${\cal F}^\bullet(X)$, contravariant under surjective maps $X \to Y$. 
We define the complex   ${\cal F}^\bullet({\rm E}\G_\bullet)$  as the total complex associated with the  bicomplex
\be
  \begin{split}
  &  \cdots \stackrel{s^*}{\longleftarrow}  \mathcal{F}^\bullet({\G^4}) \stackrel{s^*}{\longleftarrow}    \mathcal{F}^\bullet({\G^3})  \stackrel{s^*}{\longleftarrow}  
  \mathcal{F}^\bullet({\G^2})\stackrel{s^*}{\longleftarrow}  \mathcal{F}^\bullet(\G)\stackrel{s^*}{\longleftarrow}  \mathcal{F}^\bullet(\ast).\\
&s^*: = \sum_{i=0}^n (-1)^is_{n, i}^*: \mathcal{F}^\bullet({\G^{n}}) \lra \mathcal{F}^\bullet({\G^{n+1}}).\\
\end{split}
 \ee
 Applying this construction to the weight two motivic complex $\Z^\bullet_{\cal M}(\ast; 2)$,  and taking the $\G-$invariants, we get the   complex  
$$
 \Z_{\cal M}(\B\G_\bullet; 2)  :=  \Z_{\cal M}({\rm E}\G_\bullet; 2)^\G.
$$

Let  ${\rm N}$ be a maximal unipotent subgroup. Recall the principal affine space ${\cal A}:=\G /{\rm N}$.  

The canonical   projection $\G^n\lra {\cal A}^n$ induces a map of complexes, denoted $\varphi_{{\cal A} \to \G}$:  
\begin{displaymath}  
    \xymatrix{
 \Bigl( \ldots&    \ar[l]_{s^*~~~~~}   \ar[d]  {\Z}_{\cal M}^\bullet({ {\cal A}^4 ; 2})    &    \ar[d]  \ar[l]_{~~s^*}     {\Z}_{\cal M}^\bullet({  {\cal A}^3 ; 2})   &   \ar[d]  \ar[l]_{~~s^*}   {\Z}_{\cal M}^\bullet({  {\cal A}^2 }; 2)  &   \ar[d]  \ar[l]_{~~s^*} {\Z}_{\cal M}^\bullet({ {\cal A} }; 2)  \Bigr)^\G \\
  \Bigl(\ldots& \ar[l]_{s^*~~~~~}     {\Z}_{\cal M}^\bullet({ G^4 }; 2)    & \ar[l]_{~~s^*}       {\Z}_{\cal M}^\bullet({  \G^3 }; 2)   &     \ar[l]_{~~s^*}   
 {\Z}_{\cal M}^\bullet({  \G^2 }; 2)  &   \ar[l]_{~~s^*}   
   {\Z}_{\cal M}^\bullet({ \G}; 2) \Bigr)^\G \\}
         \end{displaymath} 
 
We define a degree $4$ cycle in the total complex associated with the    bicomplex illustarted on the diagram. 
 \begin{figure}[!h]
\hspace{-0cm}
\begin{tikzpicture}
\centering
  \matrix (m) [matrix of math nodes,row sep=2.5em,column sep=2em,minimum width=1em]
  {
    & \ldots  & \ldots &
     \bigoplus\limits_{D\in {\rm X}_2({\rm Conf}_2({\cal A}) )}\mathbb{Z}   & \\     
      & \ldots  & \bigoplus\limits_{D\in {\rm X}_1({\rm Conf}_3({\cal A}))}\mathbb{Q}(D)^\times &
     \bigoplus\limits_{D\in {\rm X}_1({\rm Conf}_2({\cal A}))}\mathbb{Q}(D)^\times   &  \\     
      &\bigwedge^2\mathbb{Q}({\rm Conf}_4({\cal A}))^\times      & \bigwedge^2\mathbb{Q}({\rm Conf}_3({\cal A}))^\times  &
     \bigwedge^2\mathbb{Q}({\rm Conf}_2({\cal A}))^\times &  \\
    & B_2\Bigl(\mathbb{Q}({\rm Conf}_4({\cal A}))\Bigr)       & B_2\Bigl(\Q({\rm Conf}_3({\cal A}))\Bigr)  & 
     B_2\Bigl(\mathbb{Q}({\rm Conf}_2({\cal A}))\Bigr) &  \\     
};
  \path[-stealth]     
\foreach \row in { 1,2,3,4} {
    \foreach \col/\colnext in { 2/3,3/4}{
        (m-\row-\colnext) edge node [above] {$s^*$} (m-\row-\col)
    } 
}
\foreach \col in { 2,3,4}{
        (m-2-\col) edge node [left] {${\rm div}$} (m-1-\col)
        (m-3-\col) edge node [left] {${\rm res}$}(m-2-\col)
     (m-4-\col) edge node [left] {$\delta$}(m-3-\col)        
    };
\draw[rotate=19, color = red] (0,-0.9) circle [x radius=7.0cm, y radius=0.9cm];
\end{tikzpicture}
\end{figure}
It is given by    the    encircled in the bicomplex degree $4$ cocycle ${\rm C}^\bullet = ( {\rm C}^{(1)},  {\rm C}^{(2)},  {\rm C}^{(3)})$:
 \be
\begin{split}
& {\rm C}^{(1)} \in \B_2\Bigr({\Bbb Q}({\rm Conf}_4({\cal A}))\Bigl),~~~~~ 
  {\rm C}^{(2)} \in \bigwedge^2  {\Bbb Q}({\rm Conf}_3({\cal A}))^\times, ~~~~~ 
  {\rm C}^{(3)} \in \bigoplus\limits_{D\in {\rm X}_1({\rm Conf}_2({\cal A}) ) }{\cal O}^\times_D.~~~~~~~~~~~\\
\end{split}
  \ee
 The  cocycle property     just means that
  \be \la{COSC}
 s^*({\rm C}^{(1)})=0, ~~~~ \delta({\rm C}^{(1)}) = s^*({\rm C}^{(2)}),
~~~~ {\rm res}({\rm C}^{(2)}) = s^*({\rm C}^{(3)}), ~~~~{\rm div}  ({\rm C}^{(3)})=0. 
\ee
    The cocycle will be well defined up to a coboundary. It provides a cocycle $\varphi_{{\cal A} \to \G}({\rm C}^\bullet )$.

\bt \la{Th1.3} There is  a cocycle  ${\rm C}^\bullet =  ( {\rm C}^{(1)},  {\rm C}^{(2)},  {\rm C}^{(3)}) $ such that the induced cocycle $\varphi_{{\cal A} \to \G}({\rm C}^\bullet )$ represents 
 the   second  motivic Chern class
 \be \la{C2}
 {\rm c}_2 \in \rH^4(\B\G_\bullet, \underline \Z_{\cal M}(2)).
 \ee
 \et

   If $\G = {\rm SL}_m$, there is a   canonical cocycle ${\rm C}^\bullet$,  defined in \cite{G93}.  Given a non-trivial representation $V$ of the  group $\G$, the 
  pull back of this cocycle via the embedding $\G \hra {\rm SL(V)}$ is a non-trivial cocycle for $\G$. However  in general we can not get the  generator of the group $\rH^4$ this way. 
  For example, for the group of type $E_8$, the closest we get this way is $60\cdot {\rm c}_2$ for the adjoint representation.

  \paragraph{11. Cluster nature of the construction.} Our construction is cluster.  The construction of the components ${\rm C}^{(1)},  {\rm C}^{(2)}$   uses essentially  
   the construction of the cluster   structure on the   moduli space 
 ${\cal A}_{\G, \bS}$ \cite{FG1},  closely related to the moduli space of  $\G-$local systems on   a   decorated surface $\bS$, in the case when $\bS$ is a triangle or a quadrilateral. 
 For $\G = {\rm SL}_m$ this  is explained in \cite[Section 12]{FG1}. 
 
  On the other hand, the construction of the cluster structure for the general moduli space ${\cal A}_{\G, \bS} $ follows immediately from the one for a triangle and rectangle, 
  provided that we prove that these cluster structures are invariant under the twisted cyclic rotations of these polygons. The latter is the most challenging part of the proof in \cite{GS19}, 
  which takes about 30 pages of elaborate calculations, with the final result coming as a pleasant surprise. 
  Our approach  explains why the cluster structure should be invariant under the twisted cyclic  shift, and establishes a key step of the proof without any elaborate computations. 
   
 The last component ${\rm C}^{(3)}$ is crucial to prove that   the  cohomology class  $[{\rm C}^\bullet]$ coincides with the motivic Chern class ${\rm c}_2$.

\paragraph{12. Applications.} This construction has numerous   applications. Here are some of them. 
  
  \begin{enumerate}
  
  \item 
  
  An explicit construction on the level of cocycles of the universal extension of the group $\G$ by $K_2$. 

Thus we get an explicit construction of the Kac-Moody group $\widehat \G$ given by a central extension of the loop group:
\be
1 \lra {\Bbb G}_m \lra \widehat \G \lra \G((t)) \lra 1.
\ee

\item We  get  an explicit construction of the  line bundle   
    generating   the Picard group of ${\rm Bun}_\G(\Sigma)$, where $\Sigma$ is a Riemann surface with punctures.     See  \cite{LS} for the background  
    on the generating line bundle.

\item Using the dilogarithm and the weight two exponential complex \cite{G15}, we get an explicit combinatorial formula
 for the second Chern class of a $\G-$bundle on a manifold, with values  in the Beilinson-Deligne complex.  
 In particular we get a combinatorial formula for the  second integral  Chern class, in the spirit of the  Gabrielov-Gelfand-Losik  combinatorial formula  \cite{GGL} for the first Pontryagin class.

\item Given a  punctured surface $S$, let  ${\cal U}_{\G, S}$ be the moduli space parametrizing   {\it framed} unipotent $\G-$local systems on $S$, that 
is  $\G-$local systems  with unipotent monodromies around the punctures,  equipped with a reduction to the Borel subgroup at each puncture.

Let ${\rm M}$ be a threefold whose boundary  is the surface $\overline S$  \underline{with} filled punctures. We prove that the subspace ${\cal M}_{\G, M}\subset {\cal U}_{\G, S}$  parametrising  framed unipotent 
$\G-$local systems on $S$ which can be extended to  
 ${\rm M}$ is a $K_2-$Lagrangian. 
We define the {\it motivic volume} map on its generic part
 \be
{\rm Vol}_{\rm mot}:  {\cal M}^\circ_{\G, {\rm M}} \lra \B_2(\C) 
  \ee
 valued  in the Bloch group of $\C$. Its composition with the map  $\B_2(\C)\to \R$  provided by the Bloch-Wigner dilogarithm  is a volume map  generalising the volume of a hyperbolic threefold. 
  For $\G={\rm GL_m}$ these results were obtained in \cite{DGG}  using the canonical cocycle for ${\rm GL}_m$.
  \end{enumerate}
 
  \paragraph{13. Quantum deformation of measurable cocycles of $\G(\C)$.} The measurable cohomology 
  ${\rm H}^{*}_{\rm meas}(\G, \R)$ of a Lie group $\G$ are   the cohomology of the complex of  $\G-$invariants of  measurable functions on ${\rm Meas}(\G^n$):
\be
\begin{split}
  \ldots \lra  & {\rm Meas}(\G^{n-1})^\G  \lra   {\rm Meas}(\G^n)^\G  \lra    {\rm Meas}(\G^{n+1})^\G  \lra \ldots \\
&df(g_1, ..., g_n):= \sum_{i=1}^n(-1)^if(g_1, ..., \widehat g_i, ..., g_n).\\
\end{split}
\ee
  
Let $\G_\C=\G(\C)$. The algebra ${\rm H}^{*}_{\rm meas}(\G_\C, \R)$ is a graded commutative algebra generated by the classes 
 \be \la{cohcl}
 b_{\G, 2d_m-1} \in {\rm H}^{2d_m-1}_{\rm meas}(\G_\C, \R).
\ee
 where $\{d_m\}$ are the exponents of $\G$, that is  degrees of the generators 
of the   ring ${\rm S}^*(\mathfrak{h})^W$.   
One   has $d_1=2$. 
 So  $b_{\G, 3}   = b_3$. 
 For example, when $\G$ is of type 
 ${\rm A}_r$, we have $(d_1, ..., d_r) = (2, 3, ..., r)$. \\
 
The   space ${\rm Conf}_n({\G/\B})$ of $\G-$orbits on $(\G/\B)^n$ has a cyclically invariant cluster Poisson structure \cite{GS19}. Therefore it gives rise to an algebra of q-deformed functions ${\cal O}_q({\rm Conf}_n({\G/\B}))$. The  maps 
 $$
 s_i: {\rm Conf}_{n+1}(\G/\B) \lra {\rm Conf}_{n}(\G/\B), ~~~~(\B_1, ..., B_{n+1}) \lms (\B_1, ..., \widehat \B_i, ... , B_{n+1}) \ $$
 are cluster Poisson, and  
 give rise to the maps of algebras
 $$
 s_i^*: {\cal O}_q({\rm Conf}_n(\G/\B)) \lra  {\cal O}_q({\rm Conf}_{n+1}({\G/\B})).
 $$
  
 Similarly, the space   $\G\backslash \G^n$ has a cyclically invariant cluster Poisson structure \cite{GS19}, and thus gives rise to an algebra of q-deformed functions ${\cal O}_q(\G\backslash \G^n)$. There is a  natural map
$$
\pi^*: {\cal O}_q({\rm Conf}_n({\G/\B})) \lra {\cal O}_q(\G\backslash \G^n).
$$

\bcon \la{QCCON}There exist   elements 
\be \la{BCLAS}
{\bf B}_{\G, 2d_m-1} \in {\cal O}_q({\rm  Conf}_{2d_j}({\G/\B}))
\ee
such that:

\begin{itemize}

\item They satisfy a multiplicative $(2d_m+1)-$term cocycle relation in the following form: 
$$
\prod_{j=1}^{2d_m+1}s_{2j+1}^*({\bf B}_{\G, 2d_m-1}) = 1\qquad  j \in \Z/(2d_m+1)\Z.
$$

\item The  quantum deformations of  the cohomology classes (\ref{cohcl})  are    the pull backs of the classes (\ref{BCLAS}):
 $$
 {\cal B}_{2d_m-1}:= \pi^* {\bf B}_{\G, 2d_m-1}.
  $$
  \end{itemize}
  
 \econ
 
  \bt \la{QCTH} There is the element ${\bf B}_{\G, 3}$ providing the quantum deformation of the class $
 b_3 \in  {\rm H}^{3}_{\rm meas}(\G, \R)$.
  \et
  
 Theorem (\ref{QCTH})  is proved in    Section \ref{S9}. \\

 The simplest unknown quantum cohomology class 
    is the quantum deformation of the class 
 $$
 b_5 \in  {\rm H}^{5}_{\rm meas}({\rm PGL}_3(\C), \R). 
 $$
 This class  was defined in \cite{G91} by the following function on configurations of $6$ points $(x_1, ..., x_6)$ in $\C{\rm P}^2$: 
 \be \la{triple}
\beta_5(x_1, \ldots , x_6):= {\rm Alt}_6   {\cal L}_3\left(\frac{\Delta(1,2,3)\Delta(2,3,4)\Delta(3,1,5)} {\Delta(1,2,4)\Delta(2,3,5)\Delta(3,1,6)}\right).  
 \ee
 Here ${\cal L}_3$ is the single-valued version of the trilogarithm function, and $\Delta(i,j,k):= \langle \Omega_3, l_i\wedge l_j\wedge l_k\rangle$ 
 where $l_i \in \C^3-\{0\}$ 
  lifts the point $x_i$, and $\Omega_3$ is a volume form in $\C^3$. The function $\beta_5$ satisfies the relation
 $$
 \sum_{i=1}^7(-1)^i \beta_5(x_1, \ldots , \widehat x_i, \ldots , x_7)=0.
  $$
 The $5-$cocycle is defined by 
 $$
b_5(g_1, \ldots , g_6):= \beta_5(g_1 \cdot x,  \ldots , g_6\cdot x), ~~~~x\in \C{\rm P}^2, ~~g_i \in \G_\C.
 $$
 The cocycle $b_5$ extends to a cocycle on ${\rm GL}_m$,  $m>3$ by using configurations of partial flags, see \cite{G93}. \vskip 2mm

 It became clear   later that the mysterious triple ratio  in the formula (\ref{triple}) is nothing else but 
 a cluster Poisson coordinate on the 
 moduli space ${\rm Conf}_6({\rm P}^2)$ parametrising   $6-$tuples  points on ${\rm P}^2$ modulo the action of ${\rm PGL}_3$, which is a cluster Poisson variety of the finite type $D_4$. This suggests strongly that there is a quantum deformation 
 of the element $\beta_5$ provided by an element 
 $$
 {\bf B}_5 \in {\cal O}_q({\rm Conf}_6({\rm P}^2)), 
 $$
  as well as similar $5-$cocycles on ${\rm GL}_m$, $m>3$.  Conjecture \ref{QCCON} is a generalization of this. \vskip 2mm

  The main difference between   classical and   quantum cocycles is that the latter is a sum of   commutative expressions, while the former is an ordered product  of non-commuting 
  expressions.   
The order is crucial, and provided by the cluster Poisson transformation describing the flip of a triangulation \cite{GS19}. 
  
  \vskip 2mm  
 \paragraph{Acknowledgment.}  This work was supported by the NSF grants DMS-1900743, DMS-2153059.  The first author is grateful to Linhui Shen for useful discussions. We are very grateful to the referee, who carefully red the text, and made   a lot of useful comments and suggestions which improved the exposition. 
 
 \section{The simplest example: $\G = {\rm SL}_2$} \la{SECT2}
  
The  cocycle ${\rm C}^{(\bullet)}$ for the generator of $\rH^4(\B_{{\rm SL_2} \bullet},  \mathbb{Z}_{{\mathcal{M}}}(2))$ has  three components. Using $\G= {\rm SL}_2$, they are:
\be
\begin{split}
&{\rm C}^{(1)} \in  \B_2\Bigl(\mathbb{Q}(\G^4 )\Bigr)^{{\G}},\\
&{\rm C}^{(2)} \in  \Bigl(\mathbb{Q}(\G^3 )^\times \bigwedge  \mathbb{Q}(\G^3 )^\times \Bigr)^{{\G}},\\
 &{\rm C}^{(3)}\in \Bigl(\mathbb{Q}(D)^\times\Bigr)^\G, ~~~~~D\in {\rm div}(\G^2)^\G.\\
\end{split}
 \ee

 Fix  a complex two dimensional   vector space $V_2$ with an area form $\Delta$. 
Then a flag  is   a 1-dimensional subspace of $V_2$, and a decorated flag is a non-zero vector $v \in V_2$. 
Two  decorated flags are   in generic position if $\Delta(v_1 v_2) \neq 0$. 
To construct a cocycle we pick a non-zero vector $v \in V_2$. 

\paragraph{The cycle ${\rm C}^{(3)}$.} There is   $\G-$invariant divisor 
\be
D_v\subset \G^2, ~~~~~~~D_v:= \{(g_1, g_2)\in \G^2 ~|~ \Delta(g_1v, g_2v)=0\}.
\ee
 It carries a function  
 \be
 \lambda_v(g_1, g_2):= \frac{g_1v}{g_2v}, ~~~~~~(g_1, g_2)\in D_v \subset \G^2.
 \ee
Note that the residue of this function is equal to zero. So we set
\be
{\rm C}^{(3)}:=  (D_v, \lambda_v).
\ee
The $\G-$invariant divisor with a function $(D_v, \lambda_v)$ in $\G^2$ is  the same thing as a divisor with a function $( D'_v, \lambda'_v)$ for the quotient   $\G^2/\G = \G$. 
Namely, we  identify $\G$ with the section $\{e\} \times \G\subset\G^2$. 

 To check that the   current $2 \pi i \cdot d\log ( \lambda'_v) \delta(D'_v) $    generates   $\rH^3({\rm SL}_2(\mathbb{C}), \mathbb{Z}(2))$, we integrate  it over the cycle generating  the 3-dimensional homology of ${\rm SL}_2(\mathbb{C}) $, given by the subgroup  $SU(2)$. Precisely, pick a Hermitian form $\langle, \rangle$ in $V_2$ and an orthonornal basis   $(v,w)$ containing $v$. Then 
\be
\begin{split}
&SU(2) = \left(\begin{matrix} \alpha & \beta \\ -\overline{\beta} & \overline{\alpha} \end{matrix}\right) ~~~~~~\alpha, \beta \in \mathbb{C}, ~~|\alpha|^2 + |\beta|^2 = 1.\\
&D'_v=  \left(\begin{matrix} a^{-1} & b \\ 0  &  a  \end{matrix}\right), ~~~~D'_v\cap SU(2) =  \left(\begin{matrix} \alpha & 0 \\ 0  &  \overline \alpha  \end{matrix}\right), |\alpha|=1, ~~~~\lambda'_v=\alpha.\
\end{split}
\ee 
Integrating the current over   $ SU(2)$ we get 
$  2\pi i \cdot  \int\limits_{|\alpha| = 1} d\log \alpha = (2\pi i)^2.$
So  its cohomology class   generates the group
 $\rH^3({\rm SL_2}(\mathbb{C}) , \mathbb{Z}(2)).$ 

\paragraph{The component ${\rm C}^{(2)}$.} Below we use the notation $v_i:= g_iv$. We define ${\rm C}^{(2)}$ by setting 
\be
\begin{split}
&{\rm C}^{(2)}  \in  \Bigl(\mathbb{Q}({\rm SL}_2^3)^\times \wedge \mathbb{Q}({\rm SL}_2^3)^\times \Bigr)^{{\rm SL_2}}\\
&{\rm C}^{(2)} = \Delta(v_1 v_2) \wedge \Delta(v_1 v_3) + \Delta(v_1 v_3) \wedge \Delta(v_2 v_3) + \Delta(v_2 v_3) \wedge \Delta(v_1 v_2).\\
\end{split}
\ee
 Let us compute the residue    of ${\rm C}^{(2)} $. The divisors  supporting the residue  are:
$$
D_{ij} = \{\Delta(v_i v_j) = 0\}.
$$ 
The residue of ${\rm C}^{(2)} $ at the divisor $D_{12}$ is 
\be
\begin{split}
&{\rm res}_{\Delta(v_1v_2)=0}({\rm C}^{(2)} ) = {\rm res}_{\Delta(v_1v_2)=0}\Bigl(  \Delta(v_1v_2) \wedge \frac{\Delta(v_1v_3)}{\Delta(v_2v_3)}\Bigr) 
  =  {\frac{\Delta(v_1v_3)}{\Delta(v_2v_3)}}=  \Bigl({\frac{v_1}{v_2}}\Bigr) = (D_{12}, \lambda_{1/2}).  \\
  \end{split}
\ee
The result  does not depend on $v_3$ since on the divisor $  \{\Delta(v_1v_2) = 0\} $ the vectors $v_1$ and $v_2$ are parallel. 
The total   residue is 
$$
{\rm res}({\rm C}^{(2)} ) =  (D_{12}, \lambda_{1/2}) + (D_{23}, \lambda_{2/3}) +  (D_{31}, \lambda_{3/1}) = s^* {\rm C}^{(3)}.
$$
It  splits into three parts, one  for each edge of the triangle. So we can set

\paragraph{The component ${\rm C}^{(1)}$.} 
Consider  the cross-ratio 
\be
{\rm C}^{(1)}:=\{r_2(v_1,v_2,v_3,v_4)\}_{B_2} = \left\{-\frac{\Delta(v_1 v_2){\Delta(v_3  v_4)} }{\Delta(v_1  v_4)\Delta(v_2  v_3)}   \right\}_{B_2}.
\ee  
The 5-term relation in the definition of the Bloch group implies that  $s^*{\rm C}^{(1)}=0$.

The key step is the calculation of the    differential in the Bloch complex: 
\be
\la{OL}
\delta{\rm C}^{(1)} = \delta \{r_2(v_1,v_2,v_3,v_4)\}  =- \frac{1}{2} {\rm Alt}_4 \Bigl(\Delta(v_1 v_2) \wedge \Delta(v_1 v_3)\Bigr).
\ee
where $ {\rm Alt}_4$ means that we take the alternating sum over all permutations of vectors $v_1,v_2,v_3,v_4$. 
\begin{figure}[!h]
\begin{tikzpicture}[line join = round, line cap = round]
\pgfmathsetmacro{\factor}{1};
\tikzstyle{arrow} = [thick,->,>=stealth];
\coordinate [label=above:$v_1$] (A) at (4,4);
\coordinate [label=right:$v_2$] (B) at (6,-2);
\coordinate [label=left:$v_3$] (C) at (-2,-2);
\coordinate [label=below:$v_4$] (D) at (4,-4);
\coordinate [label=right:$\Delta (v_1v_2)$] (E) at (5,1);
\coordinate [label=below:$\Delta (v_2v_3)$] (F) at (2,-2);
\coordinate [label=below:$\Delta (v_3v_4)$] (G) at (1,-3);
\coordinate [label=left:$\Delta (v_1v_3)$] (H) at (1,1);
\coordinate [label=right:$\Delta (v_2v_4)$] (I) at (5,-3);
\coordinate [label=right:$\Delta (v_1v_4)$] (J) at (4,-1);

\draw[-, opacity=.5] (A)--(B)--(D)--(C)--cycle;
\draw[arrow, opacity=1] (E)--(H);
\draw[arrow, opacity=1] (H)--(J);
\draw[arrow, opacity=1] (J)--(E);
\draw[arrow, opacity=1] (H)--(F);
\draw[arrow, opacity=1] (F)--(E);
\draw[arrow, opacity=1] (G)--(H);
\draw[arrow, opacity=1] (J)--(G);
\draw[arrow, opacity=1] (E)--(I);
\draw[arrow, opacity=1] (I)--(J);
\draw[arrow, opacity=1] (F)--(G);
\draw[arrow, opacity=1] (G)--(I);
\draw[arrow, opacity=1] (I)--(F);

\draw[-, opacity=.5] (B)--(C);
\draw[dashed, opacity=.5] (A)--(D);
\end{tikzpicture}
\caption{Calculating $\delta{\rm C}^{(1)} $ for the group ${\rm SL_2}$, and  the octahedron.}
\end{figure}
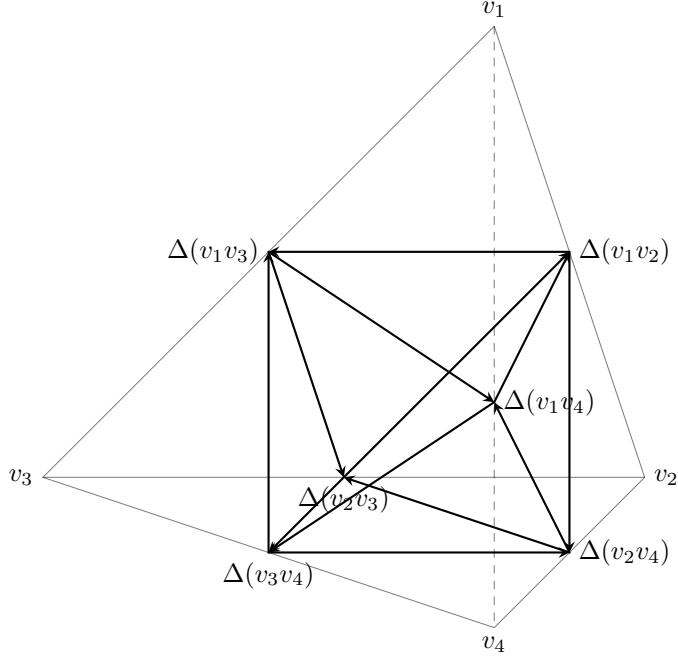

We picture $\delta{\rm C}^{(1)} $ on  Figure 3 as a 3-dimensional simplex with four flags at the vertices, and  elements
 $\Delta(v_iv_j)$  at the centers of the corresponding edges. 
 Each arrow represents a summand in   (\ref{OL}). For example the arrow $\Delta(v_1v_2) \to \Delta(v_1v_3)$ represents  
$ \Delta(v_1 v_2) \wedge \Delta(v_1 v_3)$.  
The terms in  (\ref{OL})  split into parts that live  on the   faces, i.e.  depend only on three flags.

  \section{The  components  
  ${\rm C}^{(1)},  {\rm C}^{(2)}$ of the cocycle} \la{SECT4}

\paragraph{1. Cluster varieties set-up.} 
Let us recall quivers, also known as seeds, see  \cite[Definition 1.4]{FG03b}.
\begin{definition} A quiver ${\bf c} $ is a data $ \{\Lambda, (*,*), \{e_i\}, \{d_i\}, i \in {\rm I}, {\rm I}_0\subset {\rm I}\}$, where:
\begin{itemize}
\item $\Lambda$ is an integral lattice; 
  $(*,*)$ is a bilinear symmetric $\Q-$valued form form on $\Lambda$; 
  $\{e_i\}$   is a basis for $\Lambda$ parametrised  by a  is a finite set ${\rm I}$ - we   call its elements vertices,  ${\rm I}_0$ is  the subset of  frozen vertices; and 
  $\{d_i\}$ is a set of positive integers such that:
 $\varepsilon_{ij} = (e_i,e_j) d_j \in \mathbb{Z}$ unless $i,j \in {\rm I}_0$,  when $\varepsilon_{ij} \in \frac{1}{2}\Z$.  
 \end{itemize}
\end{definition}

We   describe a quiver geometrically by drawing a vertex for each basis element $e_i$, and  $n = \varepsilon_{ij}$ arrows from the vertex $e_i$ to the vertex $e_j$ if $n > 0$ or in the opposite direction if $n < 0$. 

\begin{definition} For each unfrozen vertex $e_k$ of a quiver    ${\bf c}$ 
   there is a quiver mutation  $\mu_k: {\bf c} \to {\bf c}'$  defined as the change of the basis for $\Lambda$: 
$$
e_i' = \begin{cases}
-e_k,\ &i=k\\
e_i + [\varepsilon_{ik}]_{+} e_k,\ &i\neq k, ~~~~~~~~[a]_{+}:=  max(a,0). \end{cases}
$$
\end{definition}
 
Let  $\{f_i\} \in {\rm Hom}(\Lambda,  \Q)$ be  the quasidual   to $\{e_i\}$ basis: $\langle f_i, e_j\rangle = d_i^{-1}\delta_{ij}$, and  $\Lambda^{\circ}$    the sublattice generated by $\{f_i\}$. 
Consider the split torus:
$$
 \mathcal{A}_{\Lambda}  := {\rm Hom} (\Lambda^\circ, \mathbb{G}_m).
 $$
It comes with  cluster  $\mathcal{A}-$coordinates  $\{{\rm A}_i\}$ provided by the basis $\{f_i\}$. 

One associates with the   mutation $\mu_k: {\bf c} \to {\bf c}'$ a transformation of the cluster   coordinates, acting   by  
 \be \la{AA}
 \begin{split}
&\mu_k^* A_i =A_i, \ \ \ i \not = k, \\
&\mu_k^* A_k =  \frac{1}{A_k} \left(\prod\limits_{\varepsilon_{ik}>0} A_i^{\varepsilon_{ik}} + \prod\limits_{\varepsilon_{ik}<0} A_i^{-\varepsilon_{ik}} \right). \\
\end{split}
\ee

The cluster variety ${\cal A}$ with the initial quiver ${\bf c}$ is obtained by gluing the tori $\mathcal{A}_{\Lambda} $ assigned to quivers obtained from ${\bf c}$ by sequences of mutations via 
the corresponding composition of the transformations (\ref{AA}). By the Laurent Phenomena  theorem \cite{FZI},   each   element  $A_i$ is a regular function on ${\cal A}$. 
The algebra of regular functions ${\cal O}({\cal A})$ is nothing else but the Fomin-Zelevinsky upper cluster algebra. \\

Therefore each  cluster    ${\bf c}$ on a cluster variety ${\cal A}$ 
 is given by a collection of cluster coordinates $(\A_1, ..., \A_m)$ and an 
exchange matrix $\varepsilon_{ij}$ with the skewsymmetrizers $d_i$. This data  is encoded  in a single element 
\be \la{WC1}
{\rm W}_{\bf c}:= \frac{1}{2}\cdot \sum_{i,j\in {\rm I}} d_i\varepsilon_{ij} \cdot \A_i\wedge \A_j\in {\cal O}^\times({\cal A}) \wedge {\cal O}^\times({\cal A})\otimes \Z[\frac{1}{2}].
\ee
 Note that  $2\cdot  {\rm W}_{\bf c}$  has integral coefficients, and ${\rm W}_{\bf c}$ has coefficients in  $\Z$ if ${\rm I}_0$ is empty.   
 
Let us assign to a cluster mutation $\mu_k: {\bf c} \to {\bf c}'$  a rational function, written in the coordinate system $\{A_i\}$ for the cluster ${\bf c}$ as  
\be
\widehat X_k^{\bf c}:= \prod_{j\in {\rm I}}A_j^{\varepsilon_{ij}}.
\ee
Then the elements ${\rm W}_{\bf c}$ and ${\rm W}_{\bf c'}$ differ by the Steinberg relation  \cite[Proposition 6.3]{FG03b}:
\be \la{STR}
{\rm W}_{\bf c'} - {\rm W}_{\bf c} = d_k\cdot (1+\widehat X^{\bf c}_k) \wedge \widehat X^{\bf c}_k.\ee

 \paragraph{2. The  moduli space $\mathbf{\mathcal{A}}_{\G,\bS}$.}  
 Let us recall  the definition of the  moduli space  
   $\mathbf{\mathcal{A}}_{\G,\bS}$ \cite{FG1}.

\begin{definition} Let $\bS$ be a decorated surface. 
Let $\G$ be a simply-connected split semi-simple group. 

The moduli space $ {\mathcal{A}}_{\G,\bS}$   parametrises   twisted $\G$-local system  $\mathcal{L}$ on  $\bS$  together  with  
  a  flat section of the local system  $\mathcal{L} \times_{\G}\mathcal{A}$  near the special points and punctures.
\end{definition}
According to  the main result  of \cite{GS19}, the moduli space $\mathcal{A}_{\G,\bS}$ has a cluster ${\cal A}-$variety structure.

In particular, when the decorated surface $\bS$ is an oriented $n-$gon ${\rm p}_n$, we get the space 
$$
\mathcal{A}_{\G, {\rm p}_n} = {\rm Conf}_n({\cal A}):= \G \backslash {\cal A}^n, ~~~~~~{\cal A}:=\G/\U.
$$
The isomorphism depends on the choice of a vertex of the polygon. For example,  for the triangle $t$:  
\begin{itemize}
\item The space $\mathcal{A}_{\G,t}$ is the configuration space of   three decorated flags    - ${\rm Conf}_3(\mathcal{A})$.
\end{itemize}

\paragraph{3. An element  $ {\rm C}^{(2)}$.} Pick a reduced decomposition  of the longest element $w_0$ of the Weyl group:
 $$
 {\bf i}= (i_1, ..., i_n), ~~~~w_0=s_{i_1} ... s_{i_n}.
 $$
 In \cite{GS19}, there is a construction of the cluster coordinate system on the space ${\rm Conf}_3({\cal A})$,  given by  a collection of the regular functions, called the cluster coordinates 
 \be
( \A_1, \ldots , \A_m), ~~~~\A_i \in {\cal O}({\rm Conf}_3({\cal A})):= {\cal O}({\cal A}^3)^\G,
 \ee
 together with the exchange matrix $\varepsilon_{ij}\in \frac{1}{2}\Z$. We recall the construction of the cluster assigned to the reduced decomposition of ${\bf i}$ in Section \ref{SECt4}. 
Then the element ${\rm C}^{(2)}$ is defined (see  Definition \ref{4.6})  by 
 \be
 {\rm C}^{(2)}:= {\rm W}_{\bf c}=  \frac{1}{2}\cdot  \sum_{i,j} d_i\varepsilon_{ij} \cdot \A_i\wedge \A_j. 
 \ee

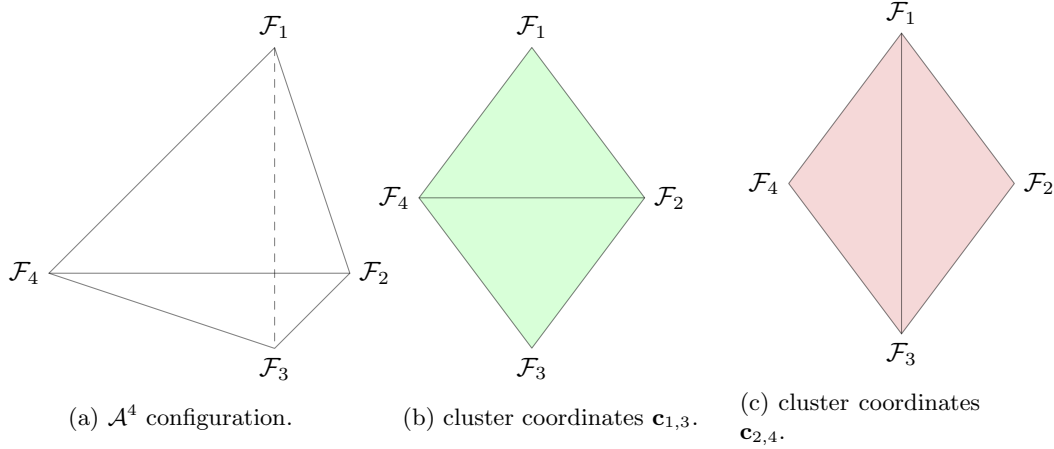
  \begin{figure}[!h]
\begin{subfigure}{0.3\textwidth}
\begin{tikzpicture}[line join = round, line cap = round]
\pgfmathsetmacro{\factor}{1};
\coordinate [label=above:$ {\cal F}_1$] (A) at (1,3);
\coordinate [label=right:$ {\cal F}_2$] (B) at (2,0);
\coordinate [label=left:$ {\cal F}_4$] (C) at (-2,0);
\coordinate [label=below:$ {\cal F}_3$] (D) at (1,-1);
\draw[-, opacity=.5] (A)--(B)--(D)--(C)--cycle;
\draw[-, opacity=.5] (B)--(C);
\draw[dashed, opacity=.5] (A)--(D);
\end{tikzpicture}
\caption{${\cal A}^4$ configuration.}
\end{subfigure}
\begin{subfigure}{0.3\textwidth}
\begin{tikzpicture}[line join = round, line cap = round]
\pgfmathsetmacro{\factor}{1};
\coordinate [label=above:$ {\cal F}_1$] (A) at (0,2);
\coordinate [label=right:$ {\cal F}_2$] (B) at (1.5,0);
\coordinate [label=left:$ {\cal F}_4$] (C) at (-1.5,0);
\coordinate [label=below:$ {\cal F}_3$] (D) at (0,-2);
\draw[-,fill=green!30, opacity=.5] (A)--(B)--(D)--(C)--cycle;
\draw[-, opacity=.5] (B)--(C);
\end{tikzpicture}
\caption{cluster coordinates ${\bf c}_{1,3}$.}
\end{subfigure}
\begin{subfigure}{0.2\textwidth}
\begin{tikzpicture}[line join = round, line cap = round]
\pgfmathsetmacro{\factor}{1};
\coordinate [label=above:$ {\cal F}_1$] (A) at (0,2);
\coordinate [label=right:$ {\cal F}_2$] (B) at (1.5,0);
\coordinate [label=left:$ {\cal F}_4$] (C) at (-1.5,0);
\coordinate [label=below:$ {\cal F}_3$] (D) at (0,-2);
\draw[-,fill=red!30, opacity=.5] (A)--(B)--(D)--(C)--cycle;
\draw[-, opacity=.5] (A)--(D);
\end{tikzpicture}
\caption{cluster coordinates ${\bf c}_{2,4}$.}
\end{subfigure}

\end{figure}

 \paragraph{4. An element $ {\rm C}^{(1)}$.}  
 Consider  two  cluster coordinate systems ${\bf c}_{1,3}$ and ${\bf c}_{2,4}$ on the space  ${\rm Conf}_4({\cal A})$: 
 
 1. The one ${\bf c}_{2,4}$, obtained by amalgamating   triangles $({\cal F}_1, {\cal F}_2, {\cal F}_3)$ and   $({\cal F}_1, {\cal F}_3, {\cal F}_4)$.
 
 2. The one ${\bf c}_{1,3}$, given by amalgamating   triangles $({\cal F}_2, {\cal F}_3, {\cal F}_4)$ and    $({\cal F}_1, {\cal F}_2, {\cal F}_4)$. 
  
 According to one of the main results of \cite{GS19}, there exists   an ordered sequence of mutations $\mu_1, \ldots, \mu_n$ providing a cluster transformation between 
 the two cluster coordinate systems above. For each mutation $\mu_i$ there is a rational function $\widehat X_i$ on ${\rm Conf}_4({\cal A})$.  So we get a collection of rational functions  
  \be \la{FNCTX}
( \widehat X_1, \ldots , \widehat X_n), ~~~~\widehat X_i \in {\Bbb Q}({\rm Conf}_4({\cal A}))^\times. 
 \ee 

 \paragraph{5. The first cocycle condition.}  Thanks to (\ref{STR}), the difference of the elements ${\rm W}$  assigned to 
 the cluster coordinate systems ${\bf c}_{1,3}$ and ${\bf c}_{2,4}$ is the sum of the Steinberg relations provided by   functions (\ref{FNCTX}):  
  \be \la{WCH}
   {\rm W}_{{\bf c}_{1,3}}  -  {\rm W}_{{\bf c}_{2,4}} = \sum_{k=1}^N d_k \cdot (1+\widehat X_k)\wedge \widehat X_k.  
   \ee  
   This just means that setting 
   \be
   {\rm C}^{(1)}:= \sum_{k=1}^N d_k \cdot \{-\widehat X_k\}\in \B_2\Bigl({\Bbb Q}({\rm Conf}_4({\cal A}))\Bigr).
    \ee
  we get, at least modulo 2-torsion,  the first cocycle identity in (\ref{COSC}):
  \be \la{BCA1}
  \delta({\rm C}^{(1)}) = s^*({\rm C}^{(2)}).
    \ee

 \paragraph{6. Altering the cluster transformation.} According to \cite{GS19}, changing a reduced decomposition ${\bf i}$ we alter the chain $({\rm C}^1, {\rm C}^2, ...)$ by   a coboundary of an element of $\B_2(\mathcal{O}_{\G^3 })$. 
 
 \bt \la{GUR}
 Changing a cluster transformation  ${\bf c}_{1,3} \to {\bf c}_{2,4}$  does not affect the element $C^{(1)}$, since it is changed by a sum of the five-term relations,  modulo an order $6$ cyclic subgroup.  
 \et
 \begin{proof} Thanks to (\ref{WCH}), for a different cluster transformation ${\bf c}_{1,3} \to {\bf c}_{2,4}$ provided by a sequence of mutations associated with the functions 
 $\widehat Y_1, ..., \widehat Y_M$ we have 
 \be \la{60}
 \sum_{k=1}^N d_k \cdot (1+\widehat X_k)\wedge \widehat X_k -\sum_{k=1}^{M} d_k \cdot (1+\widehat Y_k)\wedge \widehat Y_k = 0.
 \ee
Denote by $\beta_\F$ the kernel of the   differential $\delta: \B_2(\F) \lra  \F^\times\wedge \F^\times$  in the Bloch complex (\ref{BLOCH}). 
Then identity (\ref{60}) just means that we get, modulo $2-$torsion,  an element of the group $\beta_{\Bbb F}$, 
where  ${\Bbb F}:=  {\Bbb Q}({\rm Conf}_4({\cal A}))$ is the function field on the   configuration space:
  \be \la{80q}
   \sum_{k=1}^N d_k \cdot  \{\widehat X_k\} -  \sum_{k=1}^M d_k \cdot  \{\widehat Y_k\} \   \in \   \beta_{\Bbb F}, 
 \ee
 Let  $\widetilde{\rm Tor}(\F^\times, \F^\times)$ be
  the unique non-trivial extension of the group ${\rm Tor}(\F^\times, \F^\times)$ by $\Z/2\Z$. By   Suslin's theorem \cite{S}, for any field $\F$, there is an exact sequence
 \be
 0 \lra   \widetilde{\rm Tor}(\F^\times, \F^\times)  \lra K_3^{\rm ind}({\F})\lra  \beta_\F \lra 0.
  \ee
 Note that $ \Z/2\Z= {\rm Tor}(\Q^\times, \Q^\times) =  {\rm Tor}(\Q(t_1, ..., t_n)^\times, \Q(t_1, ..., t_n)^\times) $. 
Next,  $K_3^{\rm ind}({\F}(t)) = K_3^{\rm ind}({\F})$. Therefore, since the configuration spaces are rational varieties,  the element (\ref{80q}) provides an element 
 of  $K_3^{\rm ind}({\Bbb Q})/(\Z/4\Z)$.  
  Suslin proved \cite[Corollary 5.3]{S} that the latter group is isomorphic to $\Z/6\Z$ - this uses the  Lee and Szczarba theorem  \cite{LS}.  
Therefore the element (\ref{80q}) belongs to the subgroup $\Z/6\Z$. 
 \end{proof}

\section{Cluster structure of the  space   ${\rm Conf}_3(\mathcal{A})$} \la{SECt4}

For the convenience of the reader,  we reproduce the definition of the clusters, that is cluster coordinates and quivers,  describing the cluster structure of the space ${\rm Conf}_3({\cal A})$, borrowing  the construction of the cluster coordinates  from \cite[Section 5]{GS19}, and the construction of quivers   from \cite[Section 7.2]{GS19}.

\paragraph{1. The set-up.}
Recall that $\G$ is a split semi-simple simply-connected algebraic group with the Cartan  group $\rH$, the Weyl group 
 $W$, and the Cartan matrix  $\{{\rm C}_{ij}\}_{i,j \leq r}$, 
 simple positive roots   $\alpha_i$ and   coroots   $\alpha_j^{\vee}$:
\be
\alpha_i : \rH \to \mathbb{G}_m,\ \alpha_i^{\vee} : \mathbb{G}_m \to \rH,\ \ \alpha_i \circ \alpha_j^{\vee} = {\rm C}_{ij}.
\ee
There is   a set of the fundamental weights $\Lambda_1,\dots , \Lambda_r$:
\be
\Lambda_i : \rH \to \mathbb{G}_m, ~~~~\ \ \Lambda_i \circ \alpha_j^{\vee} = \delta_{ij}.
\ee

The length and reduced decomposition of the Weyl group elements induce the Bruhat order of Bruhat cells.  
If  elements $w, w' \in W$ have   reduced decompositions such that the one for $w'$ is a substring of the one for $w$ then  $w \succ w'$. 
 If in addition   $l(w) = l(w') + 1$ then the cell $\B w' \B$ is a boundary divisor of  $\B w\B$.

A pinning for a generic pair of flags $\{\B,\B^-\}$ provides   maps $x_i : \mathbb{A}^1\to \U$ and $y_i : \mathbb{A}^1 \to \U^-$ for every simple root $\alpha_i$, where $\U$ is the maximal unipotent subgroup of $\B$ and $\U^-$ is the maximal unipotent in $\B^-$, such that each pair $x_i,y_i$ can be extended to a standard embedding $\gamma_i : {\rm SL_2} \to \G$.  
 A pinning allows   to  lift to the group $\G$ the generators of the Weyl group $W$  corresponding to simple roots:
$$\overline{{s_i}} := y_i(1) x_i(-1) y_i(1).$$
These elements satisfy the braid relations.
Therefore  we define the lift for all other elements of $W$ by using any reduced decomposition   $w= s_1 \cdot \dots \cdot s_m$, setting:
$ \overline{{w}} = \overline{{s_1}} \cdot \dots \cdot \overline{{s_m}}. 
 $
Using this,  we define the Bruhat   decomposition   of any element $g \in \G$: \be \la{BRD}
g = u  h\overline{n_{w}}   v, ~~~~~h \in \rH = \B \cap \B^-, ~~~~u,v \in U.
\ee
 Therefore any $\G-$orbit  in the space of pairs of decorated flags 
 \be
 ( {\mathcal{F}}, {\mathcal{G}}) \in {\rm Conf}_2(\mathcal{A}) = \G\backslash (\G/\U)^2 = \U\backslash \G/\U
 \ee
 has two invariants:
  the $\omega-distance$  $\omega( {\mathcal{F}}, {\mathcal{G}}) := w$,
and the  $h-distance$   $h(  {\mathcal{F}}, {\mathcal{G}}) :=h$, where $g \in \G$ is decomposed as in (\ref{BRD}).  
Each    fundamental weight $\Lambda_i$  gives rise to  a    regular function  on every  Bruhat cell:  
\be
\Delta_{i, w}(u  h\overline{n_{w}}   v):= \Lambda_i(h).
\ee

   \paragraph{2. Cluster $\mathcal{A}$-coordinates for the space  $\mathcal{A}_{\G,t} = {\rm Conf}_3(\mathcal{A})$.}   
For each reduced word ${\bf i}=(i_1, \ldots, i_m)$ of $w_0$ there are 
   chains of  distinct positive roots and coroots:
\be
\la{sequence.beta.906}
\alpha_k^{\bf i}:=s_{i_m}\ldots s_{i_{k+1}}\cdot \alpha_{i_{k}},\hskip 7mm
\beta_k^{\bf i}:=s_{i_m}\ldots s_{i_{k+1}}\cdot \alpha_{i_{k}}^\vee, \hskip 7mm k\in\{1,..., m\}.
\ee

\begin{lemma} \la{GSA}
 \cite[Lemma 5.3]{GS19}. Given any generic pair of decorated  flags $\{ {\mathcal{F}},  {\mathcal{G}}\}$,  i.e. $\omega( {\mathcal{F}}, {\mathcal{G}}) = w_0$, 
and a reduced decomposition  ${\bf i} = \{i_1, \dots. i_m\}$ of $w$,  there exists a unique chain of decorated flags  
 \be
\{  {\mathcal{F}} =  {\mathcal{F}^0} \stackrel{s_{i_1}}{\longleftarrow}  { \mathcal{F}^1} \stackrel{s_{i_2}}{\longleftarrow}    \dots  \stackrel{s_{i_m}}{\longleftarrow}  {\mathcal{F}^m} = { \mathcal{G}}\}
 \ee
 such that that for the consequtive decorated flags, counted from the right to the left, we have:
\be
\omega ( {\mathcal{F}^{k}}, {\mathcal{F}^{k-1}}) = s_{i_{k}}, ~~~~~~~~h( {\mathcal{F}^k}, {\mathcal{F}^{k-1}}) \in 
\begin{cases}
 \alpha^\vee_{i_k}({{\Bbb G}_m}),\text{ ~~if $ \beta^{\bf i}_k$ is simple},\\
1,\text{ ~~~  \ \ \ \ \  \ \ otherwise}.\\
\end{cases}
\ee
\end{lemma}
 
 We also note that
 \be
 h_k:= s_{i_m}\ldots s_{i_{k+1}}\left(h( {\mathcal{F}^k}, {\mathcal{F}^{k-1}})\right) = 
\begin{cases}
 \alpha^\vee_{i}({b_i}),\text{ ~~if $ \beta^{\bf i}_k=\alpha_i^\vee$},\\
1,\text{ ~~~  \ \ \ \ \  \  otherwise}.\\
\end{cases}
 \ee

   Recall the involution $* : {\rm I}\to {\rm I}$   such that $\alpha_{i^*}^{\vee} = -\omega_0(\alpha_{i}^{\vee})$. 
Let $w^* := \omega_0 w \omega_0^{-1}$. Then any reduced decomposition $w = s_{i_1} \cdot \dots \cdot s_{i_k}$ provides  a reduced decomposition $w^* = s_{i_1^*} \cdot \dots \cdot s_{i_k^*}$. 
Note that  $\omega_0^* =   \omega_0$.

\begin{figure}[!h]
\centering
\definecolor{qqqqff}{rgb}{0,0,1}
\definecolor{ududff}{rgb}{0.30196078431372547,0.30196078431372547,1}
\definecolor{zzttqq}{rgb}{0.6,0.2,0}
\begin{tikzpicture}[line cap=round,line join=round,>=triangle 45,x=1cm,y=1cm]
\clip(-5.170627627517941,-1.563769592209432) rectangle (5.6002501559139874,5.516336418333225);
\draw [line width=1pt] (0,4)-- (-4,0);
\draw [line width=1pt] (-4,0)-- (4,0);
\draw [line width=1pt] (4,0)-- (0,4);
\draw [line width=1pt,color=qqqqff] (-4,0)-- (-2.399839857699015,0);
\draw [line width=1pt,color=qqqqff] (-2.399839857699015,0)-- (-0.9285587628335188,0);
\draw [line width=1pt,color=qqqqff] (2.285797789462384,0)-- (4,0);
\draw [line width=1pt] (0,4)-- (-2.399839857699015,0);
\draw [line width=1pt] (0,4)-- (-0.9285587628335188,0);
\draw [line width=1pt] (0,4)-- (2.285797789462384,0);
\begin{scriptsize}
\draw [fill=zzttqq] (-4,0) circle (3.5pt);
\draw[color=zzttqq] (-4.0964987125791266,-0.29552630371658833) node {$ {\mathcal{F}_2}= {\mathcal{F}_2^0}$};
\draw [fill=ududff] (0,4) circle (2.5pt);
\draw[color=ududff] (-0.0199248669368797,4.314262344892182) node {$ {\mathcal{F}_1}$};
\draw [fill=zzttqq] (4,0) circle (2.5pt);
\draw[color=zzttqq] (3.9151769920562897,-0.30287843872234393) node {$ {\mathcal{F}_2^m}= {\mathcal{F}_3}$};
\draw[color=black] (0.35082576180450126,-0.28231438854967552) node {...};
\draw [fill=zzttqq] (-2.399839857699015,0) circle (2.5pt);
\draw[color=zzttqq] (-2.5739412548588622,-0.30287843872234393) node {$ { {F}_2^1}$};
\draw [fill=zzttqq] (-0.9285587628335188,0) circle (2.5pt);
\draw[color=zzttqq] (-1.12557065872501245,-0.29552630371658833) node {$ {\mathcal{F}_2^2}$};
\draw [fill=zzttqq] (2.285797789462384,0) circle (2.5pt);
\draw[color=zzttqq] (1.79960203784422,-0.2808220337050771) node {$ {\mathcal{F}_2^{m-1}}$};
\draw[color=qqqqff] (-3.1010016233977447,0.1603060666402598) node {$s_{i_1}^*$};
\draw[color=qqqqff] (-1.5055883271487833,0.24117955170357153) node {$s_{i_2}^*$};
\draw[color=qqqqff] (3.1336088614829896,0.2191231466863047) node {$s_{i_m}^*$};
\end{scriptsize}
\end{tikzpicture}
\caption{The $\mathcal{A}_{\G,t}$ triangle.}
\end{figure}
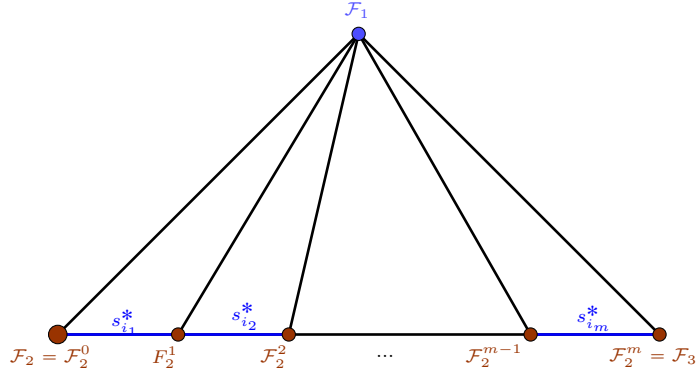

\bd \la{6.2}
The cluster  $\mathcal{A}$-coordinates on the space ${\cal A}_{\G, t}$ are defined as follows. Pick a vertex of the triangle $t$ 
with a decorated flag ${\cal F}_1$, and a reduced decomposition 
${\bf i}= (i_{1}, \ldots , i_{m})$ of $w_0$. Then: 

\begin{itemize}
	\item The \underline{frozen cluster coordinates} are:  
	$$\Delta_i( {\mathcal{F}_1},  {\mathcal{F}_2}),  ~~\Delta_i( {\mathcal{F}_1}, {\mathcal{F}_3}), ~~\Delta_i( {\mathcal{F}_3}, {\mathcal{F}_2}), ~~~~\forall i \in {\rm I}.
	$$

\item    Let ${\bf i}^*= (i^*_{1}, \ldots , i^*_{m})$. By Lemma \ref{GSA}, there is a unique chain of decorated flags, see Figure 2, 
with respect to the reduced decomposition ${\bf i}^*$:
$$\{ {\mathcal{F}_2}=  {\mathcal{F}_{23}^0} \longleftarrow  {\mathcal{F}_{23}^1} \longleftarrow  \ldots \longleftarrow   {\mathcal{F}_{23}^m} =  {\mathcal{F}_3} \}.
$$ 
Then the \underline{unfrozen  cluster coordinates} are: 
$$A_p = \Delta_{i_p}( {\mathcal{F}_1}, {\mathcal{F}_{23}^p}),$$
where $p$ runs through  indices $1, ..., m$  such that $i_p$ is not the rightmost simple reflection $i$ in ${\bf i}$,   $\forall i \in {\rm I}$. 
\end{itemize}
\ed

We stress that:

\begin{itemize}
\item Unfrozen vertices  depend on all three decorated flags;  we  picture them inside of the triangle.

  Frozen   vertices  depend only on two decorated flags;  we picture them on the sides  of the triangle.
  
  Cluster coordinates on the space ${\rm Conf}_2({\cal A})$   are labeled by the vertices $i \in {\rm I}$ of the Dynkin diagram. 
  The twisted cyclic shift   $(\mathcal{F}_1, \mathcal{F}_2) \lms (\mathcal{F}_2, s_\G \mathcal{F}_1)$  amounts to the automorphism $i \lms i^*$ of   ${\rm I}$. 
\end{itemize}

 Let us   define the quiver  ${\bf Q}({\bf i})$ for ${\rm Conf}_3({\cal A})$, assigned to the  reduced word ${\bf i} = (i_1,\ldots, i_m)$ for $w_0$.

\paragraph{3. Elementary quivers ${\bf J}(i)$.}  
Let us define the   quiver ${\bf J}(i)$, where $i \in {\rm I}$. Its underlying set is:
\be \la{SJ}
 {J}(i):= \left({\rm I}-\{i\}\right) \cup \{i_l\}\cup \{i_r\} \cup\{i_e\}.
\ee
There is a {\it decoration} map $\pi:  {J}(i)\rightarrow {\rm I}$ which sends $i_l, ~ i_r$ and $i_e$ to $i$, and is the identity map on ${\rm I}-\{i\}$. The multipliers on $ {J}(i)$ are defined by pulling back the multipliers on ${\rm I}$. 
The skew-symmetrizable matrix $\varepsilon(i)$ is indexed by $ {J}(i)\times  {J}(i)$, and  defined as follows:
\be
\label{exchange.matrix.ele.quiver}
\varepsilon(i)_{i_l,j}=\frac{- C_{ij}}{2}, ~~~\varepsilon(i)_{i_r,j}=\frac{C_{ij}}{2}, ~~~~\varepsilon(i)_{i_r,i_l}=\varepsilon(i)_{i_l,i_e} =\varepsilon(i)_{i_e, i_r}=1; \hskip 7mm \varepsilon(i)_{jk}=0 ~~\mbox{if } i\notin\{j,k\}.
\ee
 A  quiver $ {\bf J}(i)$ is pictured by a directed graph with vertices labelled by the   set (\ref{SJ}) and arrows encoding the  exchange matrices $\varepsilon=(\varepsilon_{jk})$, where 
\[
\varepsilon_{jk}= \# \{\mbox{arrows from $j$ to $k$}\} - \# \{\mbox{arrows from $k$ to $j$}\}.
\] 
Here $\# \{\mbox{arrows from $a$ to $b$}\}$ is   the {\it total weight} of the arrows from $a$ to $b$, which is a half-integer.  
The arrows from $a$ to $b$ are either dashed, and  counted with the weight $\frac{1}{2}$, 
or solid, and t  counted with the weight $1$. 
For non simply laced cases  we   use special arrows, see Example \ref{typeB3quiver}.

  \begin{example}
\la{typeB3quiver}
The quivers $ {\bf J}(1),  {\bf J}(2)$ for type $B_3$, and their amalgamation ${\bf J}(1)\ast {\bf J}(2)$,   described below:

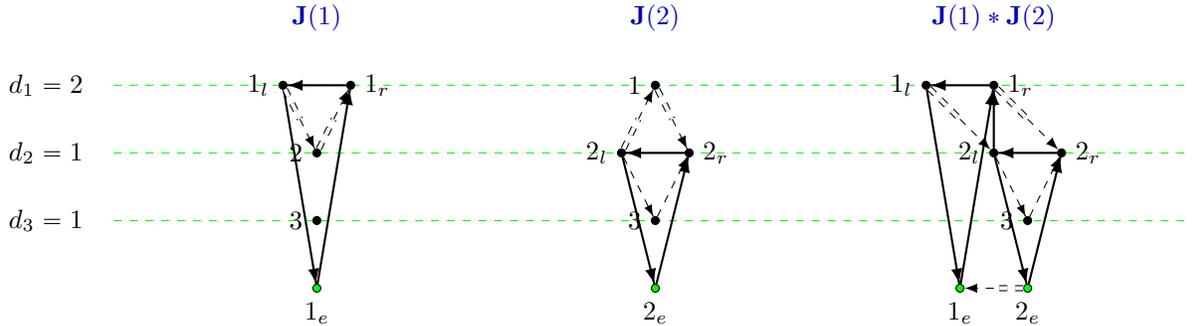
\begin{center}
\begin{tikzpicture}[scale=0.9]   
\node at (-7,0) {$d_1=2$};
\node at (-7,-1) {$d_2=1$};
\node at (-7,-2) {$d_3=1$};
         \draw[green,dashed] (-6,0) -- (10,0);     
  \draw[green,dashed] (-6,-1) -- (10,-1);
   \draw[green,dashed] (-6,-2) -- (10,-2);    
 \node [blue] at (-3,1) {$ {\bf J}(1)$};        
  \node [blue] at (2,1) {$ {\bf J}(2)$};      
  \node [blue] at (7,1) {$ {\bf J}(1) \ast {\bf J}(2)$};      
  
\node [circle,draw,fill,minimum size=3pt,inner sep=0pt, label=left:${1}_l$] (a0) at (-3.5,0) {};
\node [circle,draw,fill,minimum size=3pt,inner sep=0pt,  label=right:${1}_r$]  (b0) at (-2.5,0){};
\node [circle,draw,fill,minimum size=3pt,inner sep=0pt,  label=left:$2$] (c0) at (-3,-1) {};
\node [circle,draw,fill,minimum size=3pt,inner sep=0pt,  label=left:$3$] (d0) at (-3,-2) {};
 \node [circle,draw=black,fill=green,minimum size=3pt,inner sep=0pt, label=below: $1_e$] (e0) at (-3,-3) {};
\foreach \from/\to in {b0/a0, a0/e0, e0/b0}
                   \draw[directed, thick] (\from) -- (\to);
                   \draw[z-->] (a0) to (c0);
                   \draw[z-->] (c0) to (b0);     
                   
\node [circle,draw,fill,minimum size=3pt,inner sep=0pt, label=left:${2}_l$] (a1) at (1.5,-1) {};
\node [circle,draw,fill,minimum size=3pt,inner sep=0pt,  label=right:${2}_r$]  (b1) at (2.5,-1){};
\node [circle,draw,fill,minimum size=3pt,inner sep=0pt,  label=left:$1$] (c1) at (2,0) {};
\node [circle,draw,fill,minimum size=3pt,inner sep=0pt,  label=left:$3$] (d1) at (2,-2) {};
 \node [circle,draw=black,fill=green,minimum size=3pt,inner sep=0pt, label=below: $2_e$] (e1) at (2,-3) {};
\foreach \from/\to in {b1/a1, a1/e1, e1/b1}
                   \draw[directed, thick] (\from) -- (\to);
                   \draw[z-->] (c1) to (b1);
                   \draw[z-->] (a1) to (c1);     
                       \draw[directed, dashed] (a1) to (d1);
                   \draw[directed, dashed] (d1) to (b1);                              

\node [circle,draw,fill,minimum size=3pt,inner sep=0pt, label=left:${1}_l$] (c1) at (6,0) {};
\node [circle,draw,fill,minimum size=3pt,inner sep=0pt, label=right:${1}_r$] (c2) at (7,0) {};
\node [circle,draw,fill,minimum size=3pt,inner sep=0pt, label=left:${2}_l$] (a1) at (7,-1) {};
\node [circle,draw,fill,minimum size=3pt,inner sep=0pt,  label=right:${2}_r$]  (a2) at (8,-1){};
\node [circle,draw,fill,minimum size=3pt,inner sep=0pt,  label=left:$3$] (d1) at (7.5,-2) {};
\node [circle,draw=black,fill=green,minimum size=3pt,inner sep=0pt, label=below: $1_e$] (e1) at (6.5,-3) {};
\node [circle,draw=black,fill=green,minimum size=3pt,inner sep=0pt, label=below: $2_e$] (e2) at (7.5,-3) {};
                   \draw[directed, thick] (a2) to (a1);
                   \draw[directed, thick] (a1) to (e2);
                   \draw[directed, thick] (e2) to (a2);
                   \draw[directed, thick] (c2) to (c1);
                   \draw[directed, thick] (c1) to (e1);
                   \draw[directed, thick] (e1) to (c2);
                    \draw[directed, dashed] (d1) to (a2);     
                    \draw[directed, dashed] (a1) to (d1);     
                   \draw[z-->] (e2) to (e1);     
                    \draw[z-->] (c1) to (a1);
                   \draw[directed, thick] (a1) to (c2);
                    \draw[z-->] (c2) to (a2);

 \end{tikzpicture}
 \end{center}

\end{example}

\paragraph{4. The quiver  ${\bf H}({\bf i})$.}

 Recall   the  pairing $(\ast ,\ast)$   between the root   and coroot lattices, 
  the Cartan matrix  
$C_{ij}=(\alpha_i, \alpha_j^\vee)$, 
and the multipliers $d_j ={\langle\alpha_j^\vee, \alpha_j^\vee\rangle} \in \{1, 2,  3\}$, so that $ d_iC_{ij}$ is symmetric.

Given a reduced word ${\bf i}=(i_1, \ldots, i_m)$ of $w_0$, recall the 
   chains of  distinct positive roots $\alpha_j^{\bf i}$ and coroots $\beta_k^{\bf i}$ in (\ref{sequence.beta.906}).  
Let us define first an auxiliary quiver   ${\bf K}({\bf i})$. It    consists  of $m$   frozen vertices labeled by $(i_1, ..., i_m)$, 
with the multiplier  for the $k$th vertex given by $d_k = {\langle\alpha_{i_k}^\vee, \alpha_{i_k}^\vee\rangle} $,  and the exchange matrix   
\be \la{QV}
\varepsilon_{jk}= \left\{ \begin{array}{ll} \frac{{\rm sgn}(k-j)}{2} {(\alpha_j^{\bf i}, \beta_k^{\bf i})} &\mbox{if } i_j, i_k\in {\rm I},\\
0 &\mbox{otherwise}.\\
\end{array}
\right.
\ee

Then ${\bf H}({\bf i})$ is a full subquiver of ${\bf K}({\bf i})$ with the  vertices $k$ such that $\beta_k^{\bf i}$, and hence $\alpha_k^{\bf i}$, are simple.

   \paragraph{5. The quiver ${\bf Q}({\bf i})$.}  We use the amalgamation of quivers, introduced in \cite[Section 2.2]{FG05}.
\bd \la{6.3}
\la{quiver.amla.fg4.sec22}
Given a    reduced word ${\bf i} = (i_1,\ldots, i_m)$ for $w_0 \in W$, 
   $i_k\in {\rm I}$,  the quiver ${\bf Q}({\bf i})$ is the amalgamation of quivers $ {\bf J}(i_k)$ and ${\bf H}({\bf i})$: 
\[
{\bf Q}({\bf i}):=  {\bf J}(i_1)\ast\ldots\ast  {\bf J}(i_m)\ast {\bf H}({\bf i}).
\]

Precisely, the amalgamated quiver is defined as follows:

i) For every $i\in {\rm I}$ and for every $j=1,\ldots, m-1$, the right element of $ {\bf J}(i_j)$  at level $i$ is glued with the  left element of $ {\bf J}(i_{j+1})$ at level $i$.
The extra vertex of each $ {\bf J}(i_k)$ is glued with the $k$th vertex of ${\bf H}({\bf i})$. 

ii) The weight of an arrow obtained by gluing two arrows is the sum of the weights of those arrows. 
 \vskip 2mm

The unfrozen part of the quiver 
${\bf Q}({\bf i})$ 
is  the full subquiver  obtained by deleting the leftmost and rightmost vetices at every level $i \in {\rm I}$, and the  vertices of ${\bf H}({\bf i})$. 
\ed

The following Theorem is one of the main results of \cite{GS19}. 
  
 \bt Given a reduced decomposition ${\bf i}$ of $w_0 \in W$, the  coordinates $\{A_i\}$ from Definition \ref{6.2} and the quiver ${\bf Q}({\bf i})$  from Definition \ref{6.3}  describe an ${\cal A}-$cluster for the space 
 ${\rm Conf}_3({\cal A})$. The clusters assigned to different   reduced decompositions   are related by cluster ${\cal A}-$transformations. 
 The obtained cluster structure is invariant under the twisted cyclic shift $(\mathcal{F}_1, \mathcal{F}_2, \mathcal{F}_3) \lms (\mathcal{F}_2, \mathcal{F}_3, s_0\mathcal{F}_1)$. 
 \et

    \bd \la{4.6} Given a reduced ecomposition ${\bf i}$ of $w_0$, the element ${\rm C}^{(2)}$ is given by 
\be\begin{split}
&{\rm C}^{(2)}:= \frac{1}{2}\cdot \sum_{i,j} d_i \varepsilon_{ij} \cdot \A_i\wedge A_j,  \\
\end{split}
\ee    
where $\{A_i\}$ are the cluster coordinates from  Definition \ref{6.2}, and   $   \varepsilon_{ij}$ is the   exchange matrix for the   quiver 
${\bf Q}({\bf i})$  from Definition \ref{6.3}.
   \ed

 \section{The tame symbol  of ${\rm C}^{(2)}$ and the component ${\rm C}^{(3)}$}

Recall the tame symbol   (\ref{RESS}), also known as the residue. 
The cluster coordinates  $\{A_k\}$ are regular functions on ${\rm Conf}_3({\cal A})$. 
So for the element $ W_{\bf c}$, see (\ref{WC1}), its  tame symbol is supported  on the divisors 
$ \{ \A_k = 0\}.$ 

 The Bruhat divisor ${\cal B}_{s_kw_0} \subset {\rm Conf}_2({\cal A})$  is determined by the equation $\Delta_{k, w_0}=0$.     

Denote by $i_k$ the embedding ${\cal B}_{s_kw_0} \subset {\rm Conf}_2({\cal A})$. Recall the function $\Delta_{k,s_k\omega_0}$  on the  divisor ${\cal B}_{s_kw_0} $: 
\be
\Delta_{k, s_k\omega_0} = \Lambda_k(h_{s_k\omega_0 }(\mathcal{F}_2,\mathcal{F}_3)),~~~~~~(\mathcal{F}_2,\mathcal{F}_3)\in {\rm Conf}_2({\cal A}).
\ee
Recall   the rational function $F_k$ on  ${\cal B}_{s_kw_0}$: 
\be \la{elFk}
F_k:= i_k^*\Bigl(\Delta^{-1}_{k, s_k \omega_0}\prod_{j \in {\rm I} \setminus \{k\}}    (\Delta_{j, w_0})^{\frac{C_{kj}}{2}}\Bigr)^{d_k}.
\ee 
\bd \la{D5.1} The component ${\rm C}^{(3)}$ of the cocycle ${\rm C}^{(\bullet)}$ is defined as 
\be
{\rm C}^{(3)}:= \sum_{k \in {\rm I}}\Bigl( {\cal B}_{s_kw_0}, F_k \Bigr) \in \bigoplus\limits_{D \in {\rm div}{\rm Conf}_2({\cal A})} \mathcal{O}(D)^* .
\ee
\ed

Let $E$ be an oriented edge of the triangle $t$. Then there is a map 
$$
\beta_{E}: {\rm Conf}_3({\cal A}) \lra  {\rm Conf}_2({\cal A}).
$$ 
which forgets the element of ${\cal A}$ at the vertex of $t$ opposite to the edge $E$. 
It induces a map  
\be
\begin{split}
 \beta^*_{E} : &\bigoplus\limits_{D \in {\rm div} { {\rm Conf}}_2({\cal A})} \mathcal{O}(D)^* \to \bigoplus\limits_{D' \in  {\rm div}{\rm Conf}_3({\cal A})} \mathcal{O}(D')^*,\\
& (D, f_D) \lms (\beta^*_{E} D, \beta^*_{E} f_{D}).\\
\end{split}
\ee
We  count the vertices labeled by the  decorated flags counterclockwise: $({\cal F}_1, {\cal F}_3, {\cal F}_2)$. 
The edges $E$ of the triangle are labeled by the ordered pairs of flags $({\cal F}_{i}, {\cal F}_{j})$ assigned to them: $E = (i, j)$.

\begin{theorem} \la{BTH} The tame symbol of the element $W_{\bf c}$ on ${\rm Conf}_3({\cal A})$ is   the  sum  over the edges  of the triangle:
 \be \la{WC}
{\rm res}(W_{\bf c})  =    ( \beta^*_{1,3} +  \beta^*_{3,2} -\beta_{1, 2}^*)({\rm C}^{(3)}).
\ee

\end{theorem}

\begin{corollary}
 ${\rm div} ({\rm C}^{(3)})  = 0$.
\end{corollary}
\begin{proof}
We know that 
${\rm div} \circ  {\rm res} (W_{\bf c} ) = 0$ 
and 
\be
  {\rm div}\beta^*_{1,3}({\rm C}^{(3)}) +  {\rm div}\beta^*_{3,2}({\rm C}^{(3)})  -  {\rm div}\beta^*_{1, 2}({\rm C}^{(3)}) = 0. 
\ee
The codimension two cycles
  $ {\rm div}\beta^*_{i,j}({\rm C}^{(3)})$ can not share a common    codimension two component.  This is clear for the pull back 
to ${\cal A}^3$, since  a  point $({\cal F}_1, {\cal F}_2, {\cal F}_3)$ which lies in two cycles $ {\rm div}\beta^*_{i,j}({\rm C}^{(3)})$ satisfies codimension two condition for each of the two pars of decorated flags, which gives the codimension $>2$ intersection.  Since their sum is zero, the claim follows. \end{proof}

\begin{proof}[Proof of the Theorem]
Recall  the element $W_{\bf c}$.  
Denote by $i_k^*(f)$ the pull back of a function $f$ to the  divisor $\{\A_k=0\}$. Then    the tame symbol of $W_{\bf c}$ is  
\be
\begin{split}
&{\rm res}_{\A_k=0}(W_{\bf c}) =  {\rm res}_{\A_k=0}\Bigl( \frac{1}{2} d_i\varepsilon_{ij}\cdot \sum_{i, j} \A_i \wedge \A_j\Bigr) = i_k^*\prod_{j\not = k}  A_j ^ {d_k\varepsilon_{kj}} ,\\
&{\rm res}(W_{\bf c})  = \bigoplus_{k}  \Bigl(\{\A_k=0\}, i_k^*\prod_{j\in {\rm I} -\{k\}} \ A_j^{d_k\varepsilon_{kj}}\Bigr).\\
\end{split}
\ee
To check the last equality in the top formula here note that $d_i\varepsilon_{ij}$ is skew-symmetric, and thus we count twice 
the contribution of $ \frac{1}{2} d_k\varepsilon_{kj}\cdot A_j  = A_j ^ {d_k\varepsilon_{kj}/2}$; note the multiplicative notation used here: 
$n \cdot A=A^n$. 

There are two cases for the vertex $v_k$ related to the coordinate $\A_k$.

\paragraph{1.}   {\it The coordinate $A_k$ corresponds to a  non-frozen vertex}. This is the general case, and fortunately we can handle without going into details what is the coordinate $A_k$. 
Indeed,  since the coordinate $A_k$ is non-frozen, we can   mutate   $A_k$, getting a new cluster coordinate $\A_k'$,  
which satisfies the exchange relation: 
$$
A_k \cdot A_k' = \prod_{\varepsilon_{kj} > 0}A_j^{\varepsilon_{kj}} + \prod_{\varepsilon_{kj} < 0}A_j^{-\varepsilon_{kj}}.
$$
All we need to know is that there exists at least one non-trivial mutation at $A_k$, providing a different regular function $A_k'$ on ${\rm Conf}_3({\cal A})$. 
Restricting it to the divisor $A_k = 0$, we have 
$$ 
0 =  \prod_{\varepsilon_{kj} > 0}A_j^{\varepsilon_{kj}} + \prod_{\varepsilon_{kj} < 0}A_j^{-\varepsilon_{kj}}.
$$
Therefore 
$$ 
\prod_{j} A_j^{\varepsilon_{kj}} = -1.
$$
This is a $2-$torsion in the multiplicative group. So the residue on the divisor $ A_k = 0 $ is a $2-$torsion. 

For example, for the moduli space $\mathcal{A}_{{SL_3}, t}$ with the special cluster coordinates illustrated on the Figure, the only non-frozen coordinate is the one in the center of the triangle. 
The exchange relation  is
\be \la{D4}
\begin{split}
 &\Delta_{\omega^*}(e_1\wedge e_2, f_1\wedge f_2, g_1\wedge g_2) \Delta_\omega(e_1, f_1, g_1) = \\
 &\Delta_\omega(e_1, e_2, f_1) \Delta_\omega(f_1, f_2, g_1) \Delta_\omega(g_1, g_2, e_1) +  
  \Delta_\omega(e_1, e_2, g_1) \Delta_\omega(f_1, f_2, e_1) \Delta_\omega(g_1, g_2, f_1).\\
   \end{split}
\ee
Here $\omega$ is a volume form in a three dimensional vector space $V$, $\omega^*$ is the dual volume form in $V^*$, and ${\cal F}_1 = (e_1, e_1 \wedge e_2)$, 
${\cal F}_2 = (f_1, f_1 \wedge f_2)$ and ${\cal F}_3 = (g_1, g_1 \wedge g_2)$  are decorated flags in $V$.

\begin{figure}[!h]
\definecolor{qqqqff}{rgb}{0,0,1}
\definecolor{zzwwff}{rgb}{0.6,0.4,1}
\definecolor{qqwuqq}{rgb}{0,0.39215686274509803,0}
\definecolor{zzttff}{rgb}{0.6,0.2,1}
\noindent	
 \hspace{3cm}
\scalebox	{0.45}{
\begin{tikzpicture}[line cap=round,line join=round,>=triangle 45,x=1cm,y=1cm]
\clip(-7.5,-2) rectangle (10,10);
\draw [->,line width=1pt,dash pattern=on 1pt off 3pt] (2.1,0) -- (-2,0);
\draw [->,line width=1pt] (0.0,3.0) -- (-4,3);
\draw [->,line width=1pt] (4,3) -- (0.0,3.0);
\draw [->,line width=1pt] (2,6) -- (-2,6);
\draw [->,line width=1pt] (-2,6) -- (0.0,3.0);
\draw [->,line width=1pt] (0.0,3.0) -- (2.1,0);
\draw [->,line width=1pt,dash pattern=on 1pt off 3pt] (2,6) -- (4,3);
\draw [->,line width=1pt] (-4,3) -- (-2,0);
\draw [->,line width=1pt] (-2,0) -- (0.0,3.0);
\draw [->,line width=1pt] (0.0,3.0) -- (2,6);
\draw [->,line width=1pt] (2.1,0) -- (4,3);
\draw [->,line width=1pt,dash pattern=on 1pt off 3pt] (-4,3) -- (-2,6);
\begin{scriptsize}
\draw [fill=zzttff] (-6,0) circle (2.5pt);
\draw[color=zzttff] (-6,-0.7) node {\LARGE{$\mathcal{F}_1 = (f_1,f_2)$}};
\draw [fill=zzttff] (6,0) circle (2.5pt);
\draw[color=zzttff] (6.3,-0.7) node {\LARGE{$\mathcal{F}_1  = (g_1,g_2)$}};
\draw [fill=zzttff] (0,9) circle (2.5pt);
\draw[color=zzttff] (0,9.6) node {\LARGE{$\mathcal{F}_1 =  (e_1,e_2)$}};
\draw [fill=qqwuqq] (-2,0) circle (3pt);
\draw[color=qqwuqq] (-2.2,-0.7) node {\LARGE{$\Delta   (f_1,f_2,g_1)$}};
\draw [fill=qqwuqq] (2.1,0) circle (3.5pt);
\draw[color=qqwuqq] (2.5,-0.7) node {\LARGE{$\Delta   (f_1,g_1,g_2)$}};
\draw [fill=qqwuqq] (-2,6) circle (3.5pt);
\draw[color=qqwuqq] (-4.2,6.1) node {\LARGE{$\Delta   (e_1,e_2,f_1)$}};
\draw [fill=qqwuqq] (2,6) circle (3.5pt);
\draw[color=qqwuqq] (4.3,6.1) node {\LARGE{$\Delta   (e_1,e_2,g_1)$}};
\draw [fill=qqwuqq] (-4,3) circle (3.5pt);
\draw[color=qqwuqq] (-6.0,3.0) node {\LARGE{$\Delta  (e_1,f_1,f_2)$}};
\draw [fill=qqwuqq] (4,3) circle (3.5pt);
\draw[color=qqwuqq] (6.0,3.0) node {\LARGE{$\Delta   (e_1,g_1,g_2)$}};
\draw [fill=qqqqff] (0.0,3.0) circle (3.5pt);
\draw[color=qqqqff] (0,2.0) node {\LARGE{$\Delta (e_1,f_1,g_1)$}};
\end{scriptsize}
\end{tikzpicture}}
\caption{The canonical coordinates on the moduli space $\mathcal{A}_{{SL_3}, t}$ of triples of decorated flags.  }
\end{figure}
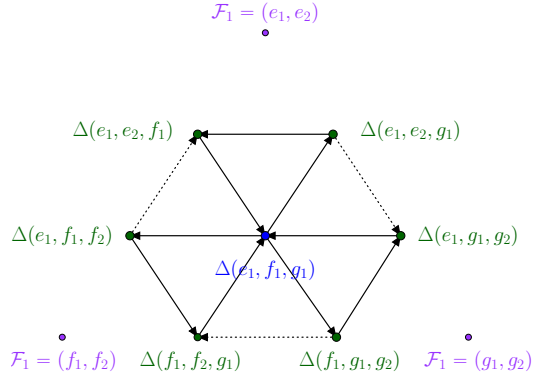

\paragraph{2.}  {\it The coordinate $A_k$ is frozen}. Then it corresponds to a  vertex located on a side of  the triangle $t$. 
This is the difficult case. Since the definition of the quiver ${\bf c}$  depends on the choice of the vertex of the triangle,   referred to as the top vertex, 
we   consider the residue for each of the three sides of the triangle. 

We   start from the right edge $({\cal F}_1, {\cal F}_3)$.  
Since the $K_2-$class  
 $[W_{\bf c}]$ does not depend on the choice of the reduced decomposition ${\bf i}$ of  $\omega_0$,  and the tame symbol depends only on    the $K_2-$class,   we can assume  that: 
 \be
 \la{DW}
 \mbox{The   decomposition ${\bf i}$    ends by  $s_k$. }
 \ee
  
\paragraph{\it The elementary configuration space $ {\cal A}({{k}})$  \cite[Section 7.5]{GS19}.} Let $k\in {\rm I}$. Consider 
the  space $ {\cal A}(k)$  parametrizing $\G$-orbits of triples of decorated flags $(\mathcal{F}, \mathcal{F}_l, \mathcal{F}_r)$ such that
\be
w(\mathcal{F}, \mathcal{F}_l)= w(\mathcal{F}, \mathcal{F}_r)=w_0,\hskip 7mm w(\mathcal{F}_r, \mathcal{F}_l)=s_{k^\ast}, \hskip 4mm h(\mathcal{F}_r, \mathcal{F}_l) \in {\rm H}(s_{k^\ast}).
\ee
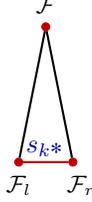
\begin{center}
\begin{tikzpicture}[scale=0.6]
\node [circle,draw=red,fill=red,minimum size=3pt,inner sep=0pt,label=above:{\small ${\mathcal{F}}$}] (a) at (0,0) {};
 \node [circle,draw=red,fill=red,minimum size=3pt,inner sep=0pt,label=below:{\small ${\mathcal{F}}_l$}]  (b) at (-0.6,-3){};
 \node [circle,draw=red,fill=red,minimum size=3pt,inner sep=0pt,label=below:{\small $~~{\mathcal{F}}_r$}] (c) at (0.6,-3) {};
 \foreach \from/\to in {b/a, c/a}
                   \draw[thick] (\from) -- (\to);
  \draw[red, thick] (c) -- (b);                 
 \node[blue] at (0,-2.7) {${ \small s_{k^\ast}}$};
  \end{tikzpicture}
 \end{center}

There is a   cluster ${\cal A}-$coordinate system on the space ${\cal A}(k)$ parametrized by $ {J}(k)$ defined by:

\be
\la{ele.cor.cl.1}
\forall (\mathcal{F}, \mathcal{F}_l, \mathcal{F}_r)\in  {\cal A}(k), \hskip 10mm
{\displaystyle{A_j := \left\{    \begin{array}{ll}  
     \Delta_j(\mathcal{F}, \mathcal{F}_l) \hskip 7mm & \mbox{if } j\in {\rm I}-\{k\}\\
       \Delta_{k}(\mathcal{F}, \mathcal{F}_l) & \mbox{if } j=k_l\\
       \Delta_{k}(\mathcal{F}, \mathcal{F}_r) & \mbox{if } j=k_r\\
      \Lambda_{k^\ast}(h(\mathcal{F}_r, \mathcal{F}_l)) & \mbox{if }j=k_e. \\
   \end{array}\right.}}
\ee
If we fix a pinning in $\G$, then we have
\be \la{Param}
(\mathcal{F}, \mathcal{F}_l, \mathcal{F}_r):= (\U_-, h \U, gh \U), \ \ \ \ h \in \rH, g \in \varphi_k({\rm SL}_2/\U_{{\rm SL}_2}).
\ee
In particular, by  \cite[Lemma 7.13]{GS19}, for any 
  $(\mathcal{F}, \mathcal{F}_l, \mathcal{F}_r)\in   {\cal A}(k)$, we have
\be
\Delta_j(\mathcal{F}, \mathcal{F}_l)= \Delta_j(\mathcal{F}, \mathcal{F}_r).  \ \ \ \ \forall j \neq k.
\ee  
There is a canonical projection
\be \la{tpr}
\tau_k: {\cal A}(k) \lra {\rm Conf}_3({\cal A}_{\rm SL_2}).
\ee
It assigns to a triple of decorated flags $(\mathcal{F}, \mathcal{F}_l, \mathcal{F}_r)$ the intersections of the corresponding maximal unipotent subgroups with the 
subgroup $\varphi_k({\rm SL}_2) \subset \G$ corresponding to the simple coroot $\alpha^\vee_k$. The coordinates $\A_{k_l}, \A_{k_r}, \A_{k_e}$ are the pull backs of the standard coordinates on  
${\rm Conf}_3({\cal A}_{\rm SL_2})$.

Recall the matrix $\varepsilon(k)$ of $ {\bf J}(k)$ in \eqref{exchange.matrix.ele.quiver},  
and  the {\it canonical}  element  describing   cluster (\ref{ele.cor.cl.1}) on $ {\cal A}(k)$: 
\be
{\rm W}(k) :=  \sum_{i,j \in  {J}(k)} d_i {\varepsilon(k)}_{ij} \cdot A_i  \wedge   A_j.
\ee

Denote by ${\rm Conf}^\times_3({\cal A})$ the subspace of ${\rm Conf}_3({\cal A})$ given by the condition that each pair of the decorated flags are in the generic position. 
The amalgamation   provides an embedding of the space ${\rm Conf}^\times_3({\cal A})$ obtained by the amalgamation   into the product of the elementary spaces ${\cal A}(i_j)$ used for the  amalgamation: 
 $$
 {\rm Conf}^\times_3({\cal A}) \hra \prod_{j=1}^m {\cal A}(i_j). 
 $$
Denote by $\eta_k: {\rm Conf}_3({\cal A})   \lra {\cal A}(k)$ the composition of this map with  the projection onto the 
rightmost factor ${\cal A}(i_j) $ with $i_j=k$.
Thanks to   assumption (\ref{DW}), the cluster coordinate $\A_k$ on   ${\rm Conf}_3({\cal A})$ is the pull back $\eta^*_k\A_k$ of the cluster coordinate $\A_k$ on   ${\cal A}(k)$. This immediately implies that   
\be
{\rm res}_{\A_k=0}(W_{\bf c} )= \eta_k^*{\rm res}_{\A_k=0}(W(k)).
\ee
So the calculation of ${\rm res}_{\A_k=0}(W_{\bf c})$ boils down to the calculation of the residue of $W_k$ at the divisor $\A_k=0$ on the elementary space ${\cal A}(k)$. 
 Let us write $W(k)$ as a sum, see (\ref{DECM}):
 \be \la{DECM}
 W(k) = W'(k) + W_\Delta(k), \ \ \ \ W_\Delta(k):= \A_{k_r} \wedge \A_{k_l} + \A_{k_l} \wedge A_{k_e} + \A_{k_e} \wedge \A_{k_r}. 
\ee
Note that by the definition of the amalgamation, $\A_k = \A_{k_l}$. 
Next, we evidently have:
 \be \la{evid}
 {\rm res}_{\{\A_k=0\}}{\rm W}'(k) = \prod_{j \not \in \{k_l, k_r, k_e\}} A_j^{d_k {\varepsilon(k)}_{kj}}.
 \ee
 Since 
$\varepsilon(k)_{k_r,j}=\frac{C_{kj}}{2}$ by (\ref{exchange.matrix.ele.quiver}),  the factor  
 $( \Delta_{j, w_0})^{\frac{d_kC_{kj}}{2}}$ 
in  (\ref{elFk})   match the factor $A_j^{d_k {\varepsilon(k)}_{kj}}$ in (\ref{evid}):
\be
( \Delta_{j, w_0})^{\frac{d_kC_{kj}}{2}} = A_j^{d_k {\varepsilon(k)}_{kj}}.
\ee
Therefore the product in (\ref{evid}) match the product over $j \not = k$ in (\ref{elFk}). So it remains to show that
\be \la{LAC}
 {\rm res}_{\{\A_k=0\}}{\rm W}_\Delta(k) = \Delta^{-1}_{k, s_k \omega_0}\ee

 Note that $W_\Delta(k) = \tau_k^*W_{{\rm Conf}_3({\cal A}_{\rm SL_2})}$, and the divisor $\A_k=0$ on ${\rm Conf}_3({\cal A})$ is the pull back  
 of the divisor $\A_{i_r}=0$ on ${\rm Conf}_3({\cal A}_{{\rm SL}_2})$ under the map $\tau_k^*$.
So    parametrisation (\ref{Param}) and the calculation of the residue for ${\rm SL}_2$ from Section \ref{SECT2}  implies (\ref{LAC}).  
So we calculated the residue for the right edge of the triangle. 

For the left edge the calculation is similar. We claim that  the residue corresponding to the left edge 
 is given by $-\beta_{1,2}^*({\rm C}_3)$, see (\ref{WC}). Indeed, this    
  agrees with the fact that  $\varepsilon_{k_l, j} = -{\rm C}_{kj}/2$ while $\varepsilon_{k_r, j} = {\rm C}_{kj}/2$ in (\ref{exchange.matrix.ele.quiver}), as well as with the calculation of the residue 
  for ${\rm SL}_2$. 
 
Computation for the bottom side $({\cal F}_3, {\cal F}_2)$ follows easily using $W_\Delta(k) = \tau_k^*W_{{\rm Conf}_3({\cal A}_{\rm SL_2})}$  and (\ref{QV}). 
  \end{proof}

\section{Proof of Theorems \ref{TH1} and \ref{Th1.3}}  \la{S2}

\paragraph{Brylinsky-Deligne  results \cite[Section 4]{BD}.} 
    
 Let   $W^{(p)}\subset W$ the subset   parametrising Bruhat cells $\B w\B$ of codinension $p$. 
    In particular, 
\be
W^{(1)} = \{w_0s_k\in W\}, k \in {\rm I}; \ \ \ \ W^{(2)} = \{w_0s_is_j\in W\},~~ i,j \in {\rm I}, ~i \not = j.
\ee 
Let $X= {\rm Hom}(\rH, \mathbb{G}_m)$ be the character  group of the Cartan group $\rH$.  
Consider the following complex
\be \la{BDC}
  \bigwedge^2 X  \lra \bigoplus\limits_{W^{(1)}} X     \lra \bigoplus\limits_{W^{(2)}}\mathbb{Z}. 
\ee
Here, using the notation $(w, -)$ for an element of $\bigoplus\limits_{W^{(1)}}X$,  the   differentials are:
\be \la{82}
\begin{split}
& x_1 \wedge x_2 \lms \sum_{i \in {\rm I}} \Bigl(\langle x_1,\alpha_i^{\vee} \rangle  \cdot (w_0s_i,  x_2) -     \langle x_2,\alpha_i^{\vee} \rangle \cdot (w_0s_i, x_1)\Bigr).\\
&  (\omega_0s_j,  x) \lms \sum_{i \not = j} \Bigl((\omega_0s_is_j,  \langle x, s_j(\alpha_i^{\vee}) \rangle)  +  (\omega_0s_js_i, \langle x, \alpha_i^{\vee}\rangle)\Bigr).\\
\end{split}
\ee

\bp There exists a   natural map of complexes\\

\hspace{0.5cm}\xymatrix@C5pc{
 \bigwedge\limits^2 X 
\ar[r]^{ }
\ar[d]
& \bigoplus\limits_{W^{(1)}} X 
\ar[r]^{ }
\ar[d]
&
\bigoplus\limits_{W^{(2)}}\mathbb{Z}
\ar[d]\\
\bigwedge^2\mathbb{Q}(G)^* 
\ar[r]^{\rm res} & 
\bigoplus\limits_{D\in \G^{(1)}}\mathbb{Q}(D)^*  
\ar[r]^{\rm val}&
     \bigoplus\limits_{ \G^{(2)}}\mathbb{Z} 
}\ep

\begin{proof} The right vertical map is induced by the canonical embedding $W^{(2)} \hra \G^{(2)}$. 
We assign to a character $\chi$ of the Cartan group and an element $w\in W$  a regular function $\chi'_{\overline{{\omega}}}$ 
on the Bruhat cell ${\cal B}_w$:
\be
\chi'_{\overline{{\omega}}}(u\overline w hv):= \chi(h). 
\ee
We warn the reader that the functions $\chi'_{\overline{{\omega}}}$, and $\chi_{\overline{{\omega}}}$  from (\ref{xw}), are not the same since $\chi_{\overline{{\omega}}}(uh\overline w  v):= \chi(h)$. 
 We need now both since \cite[Section 4]{BD}  uses $\chi'_{\overline{{\omega}}}$, while \cite{GS19} uses $\chi_{\overline{{\omega}}}$.

The two left vertical maps are given by 
\be
\begin{split}
&\chi  \wedge \psi \lms \chi'_{\overline{{\omega_0}}} \wedge \psi'_{\overline{{\omega_0}}}\\
&(w_0s_k, \chi)   \lms \chi'_{\overline{{\omega_0 s_k}}}  \\
\end{split}
 \ee
Let us prove that we get a map of complexes. 
 The    valuations of the function  $\chi'_{\overline{{\omega_0}}}$ on a divisor can be non-zero only if it is a Bruhat divisor. In this case they are calculated as follows. For every	
Bruhat divisor  $\B\omega_1\B$, we can choose reduced Weyl decompositions of $\omega$ and $\omega_1$ so that:
$\omega = \omega' s_i \omega''\text{ and }\omega_1 = \omega' \omega''$. 
 Using the valuation formula of  Demazure,  \cite[Lemma 4.2]{BD} tells:
\be
{\rm val}_{\omega_1}(\chi'_{\overline{{\omega}}}) = \langle \chi, \omega''^{-1} (\alpha_i^{\vee})\rangle.
\ee
This is consistent with  formulas (\ref{82}). The Proposition is proved.
\end{proof}

Denote by   $Y_{sc}$ the lattice generated by the simple coroots. Consider the dual lattice  $X_{sc} \subset \mathcal{O}^*(H)$. We identify 
$X_{sc} = \bigoplus\limits_{D\in W^{(1)}}\mathbb{Z} $ by 
$x \lms \sum\limits_{s_i}  {x(\alpha_i^{\vee})}{\omega_0s_i}$.
Then we identify  the complex (\ref{BDC}) with 
\be \la{MGR}
\bigwedge^2 X   \lra X_{sc} \otimes X   \lra \bigoplus\limits_{W^{(2)}}\mathbb{Z}.
\ee
By \cite[Lemma in 4.4.4]{BD} or (\ref{82}),  the  $\omega_0s_is_j-$component of the differential of $C \in X_{sc} \otimes X $   is given by:
$$
C  \lra C(\alpha_i^{\vee},\alpha_j^{\vee}) + C(\alpha_j^{\vee},s_j(\alpha_i^{\vee})).
$$
We write the right hand side via the quadratic form $Q(y):={\rm C}(y,y)$ on $Y_{\rm sc}$ and the associated symmetric bilinear form  
$
\B(\alpha_i^{\vee},\alpha_j^{\vee}) = C(\alpha_i^{\vee},\alpha_j^{\vee}) + C(\alpha_j^{\vee},\alpha_i^{\vee}).
$ Namely, using $s_j(\alpha_i^\vee) = \alpha_i^\vee - \alpha_j(\alpha_i^\vee) \alpha_j^\vee$, we get:
\be
C(\alpha_i^{\vee},\alpha_j^{\vee}) + C(\alpha_j^{\vee},s_j(\alpha_i^{\vee})) = \B(\alpha_i^\vee, \alpha_j^\vee) - \alpha_j(\alpha_i^\vee) Q(\alpha_j^\vee).
\ee
Therefore 
the element $C$ is killed by the second differential if and only if the quadratic from $Q(y)$ on $Y_{\rm sc}$ is W-invariant  \cite[Lemma 4.5]{BD}. 
So the cohomology class of a cocycle in $\bigoplus\limits_{W^{(1)}}X$ is non-trivial if and only if the corresponding quadratic form is non-zero.\\

Now let us  look at the cocycle that we constructed using the cluster coordinates:
$$
{\rm C}^{(3)} = \sum_{k \in I} d_k\cdot \Bigl(\{A_k = 0\}, ~~\Delta_{k, s_k\omega_0} ^{-1}\prod_{i \in I \setminus \{k\}}   A_{i}^{\frac{C_{ki}}{2}}  \Bigr).
$$
Here is a caveat. There are two ways to write the Bruhat decompostions of an element $g\in G$:  
$$
g = u_1 h_l \overline \omega u_2 ~~\mbox{or}~~ g = u_1  \overline \omega h_r u_2, ~~~~h_l, h_r \in \rH.
$$
Following \cite{GS19},  we defined    ${\rm C}^{(3)} $   using the left one, 
while  \cite{BD}  use   the right one.
Since  $ \overline \omega h_r    = 
    \omega(h_r)  \overline \omega$, we have    $h_l = \omega(h_r)$. 
     This   impacts our formulas as follows. 
The function $\Delta_i(g) = \Lambda_i( h_l )$   on the Bruhat cell $B\omega_0B$    is equal to the function   
$$
\Lambda_i(\omega_0 (h_r)) = \Lambda_{i^*}(h_r).
$$
 Since  
 $s_k\omega_0 = \omega_0 s_{k}^*$, and 
   $\Delta_{k,s_k \omega_0} = \Lambda_k(h_l)$ on the divisor $Bs_k\omega_0B = B\omega_0s_k^*B$, we have $h_l = s_k\omega_0(h_r)$, and so  as $s_k \alpha_{k}^{\vee} (t)=  \alpha_{k}^{\vee} (t)^{-1}$, we have 
  $$
 \Lambda_k( h_l)  = \Lambda_k(s_k \omega_0(h_r)) = \Lambda_k^{-1}(\omega_0(h_r)) = \Lambda_{k^*}^{-1}(h_r).
 $$
Since  $*$ is an involution, we have $C_{ij} = C_{i^*j^*}, d_k = d_{k^*}$. 
 So the  element ${\rm C}^{(3)} \in \bigoplus_{W^{(1)}} X$   is:
\be
\begin{split}
  {\rm C}^{(3)}= 
  & \sum_{k \in I} d_k \cdot \Bigl({\cal B}_{\omega_0 s_{k}}, ~~\Delta_{k, \omega_0s_k}\prod_{i \in I \setminus \{k\}}    (  \Delta_{i, w_0})^{\frac{{\rm C}_{ki}}{2}}      \Bigr). \\
 \end{split}
  \ee
In  the Brylinski-Deligne format 
the Bruhat divisor ${\cal B}_{w_0s_k}$   corresponds to  $\omega_0s_k\in W^{(1)}$, and  identified  with the basis element dual to the coroot $\alpha_k^{\vee}$. 
Then the element  ${\rm C}^{(3)}$ is mapped to 
\be \la{KF}
 {\rm C}^{(3)} \lms \sum_{k \in {\rm I}} d_k \cdot   \Lambda_k \otimes  \Lambda_k + \sum_{i \in {\rm I}\setminus \{k\} }\frac{d_kC_{ki}}{2}  \Lambda_k \otimes \Lambda_i.
\ee
 Since $ d_k{\rm C}_{ki}$ is symmetric, we get a symmetric tensor, which is  the Killing quadratic form   in the basis   $\Lambda_k$:
 \be \la{KF1}
 {\rm C}^{(3)} \lms   \sum_{k \in {\rm I}} d_k \cdot   (\Lambda_k )^2 + \sum_{i \not = k \in  {\rm I}  } d_kC_{ki} \cdot  \Lambda_k  \Lambda_i. \ee
The corresponding quadratic form at any simple coroot $\alpha_i^{\vee}$ corresponding to a short root  $\alpha_i$  ($d_i$ = 1) is:
$$
Q(\alpha_i^{\vee}) = h_{sc}(\alpha_i^{\vee},\alpha_i^{\vee}) = 1.
$$
 Note that if simple roots $\alpha$ and $\beta$ are not orthogonal, and $\alpha$ is not shorter than $\beta$, then $\beta(\alpha^\vee) =-1$.  
 
 Although the cocycle  ${\rm C}^{(3)}$ has half-integral coefficients, we can alter it by a coboundary and get an integral cocycle.  
 Therefore  the cohomology class $[{\rm C}^{(3)}]$   is the canonical generator. Theorem \ref{TH1} is proved. \\
 
Since we   proved  that ${\rm C}^{(\bullet)}$ is a cocycle by combining Theorem \ref{BTH}, formula (\ref{BCA1}), and Theorem \ref{GUR}, and 
 its cohomology class is a generator by Theorem \ref{TH1} and   isomorphisms (\ref{BD1}),  Theorem \ref{Th1.3} is proved.

\section{Applications}  \la{Appl}
  
 It was proved in \cite{GS19} that the class $[W]\in K_2({\rm Conf}_3({\cal A}))$     is dihedrally sign-invariant, that is 
 invariant under the cyclic shifts, and skew invariant under the permutation of two vertices $(1,2,3) \lms (2,1,3)$.
 Therefore the class $[{\rm C}^{(2)}] \in K_2({\rm Conf}^\times_3({\cal A}))$    is dihedrally sign-invariant.  This implies  the important 
 
 \bp \la{71} The element 
\be
 {\rm C}^{(1)} \in \B_2(\Q({\rm Conf}_4({\cal A}))) \otimes \Q
 \ee is sign-invariant  under the action of the permutation group $S_4$ on ${\rm Conf}_4({\cal A})$.
  \ep
  
  \begin{proof}  The dihedral sign-invariance of    $[{\rm C}^{(2)}]$ implies that $\delta {\rm C}^{(1)}$ is   dihedrally sign-invariant. Since 
   ${\rm Conf}_4({\cal A})$ is a rational variety,  the group $K_3^{\rm ind}\otimes \Q$ of its function field is the same as for $\Q$, and thus trivial.   
  \end{proof}
 
   \paragraph{1. The universal $K_2-$extension of $\G$.} Its  existence   was proved by Matsumoto, and revisited by Brylinsky-Deligne \cite{BD}. 
   However no explicit cocycle description was known before. Here is one. 
   
   Pick a decorated flag ${\cal F}\in {\cal A}_\G$. Then given a generic triple $(g_1, g_2, g_3) \in \G^3(\F)$ we set
   \be
   {\cal C}_2(g_1, g_2, g_3):= {\rm C}^2(g_1 \cdot  {\cal F},   g_2 \cdot  {\cal F},   g_3 \cdot  {\cal F}) \in K_2(\F).
   \ee
Then for any generic quadruple $(g_1, g_2, g_3, g_4)$ it satisfies the  2-cocycle condition.  
   It is well known that a 2-cocycle   of $\G(\F)$ with values in $K_2(\F)$, defined at the generic point, determines the group extension
   \be
   1 \lra K_2(\F) \lra \widehat \G(\F) \lra \G(\F) \lra 1.
   \ee
   
 \paragraph{2. The Kac-Moody group extension related to a Riemann surface.} Let $\Sigma$ be a Riemann surface with punctures. 
 Then there is a group extension
 \be
 1 \lra \rH^1(\Sigma, \C^*) \lra \widehat \G_\Sigma \lra \G({\rm Hol}(\Sigma))\lra 1.
 \ee
 Here ${\rm Hol}(\Sigma)$ stands for the field of all holomorphic functions with arbitrary singularities at the punctures, 
 including the functions with essential singularities at the punctures, e.g. $e^{c_1/z + c_2/z^2+ ...}$. 
 There is an algebraic  variant where we take  the field of rational functions on $\Sigma$ and the corresponding group $\G(\C(\Sigma))$. 
 
 The extension is the push down of the universal extension of $\G({\rm Hol}(\Sigma))$ by $K_2({\rm Hol}(\Sigma))$ by the Beilinson-Deligne regulator map
 \be
 {\rm reg}: K_2({\rm Hol}(\Sigma)) \lra  \rH^1(\Sigma, \C/\Z(2)).
 \ee
 Namely, following Beilinson \cite[Lemma 1.3.1]{Be} and Deligne \cite{D}, given   an element $f\wedge g$ and a loop $\gamma$ on $\Sigma$,  the value of the cohomology class ${\rm reg}(f\wedge g)$ on the homology class $[\gamma]$ is given by the integral
 \be
  \langle {\rm reg}(f \wedge g), \gamma\rangle  := {\rm exp}~ \frac{1}{2\pi i} \cdot \Bigl(\int_{\gamma}\log f d\log g - g(p) \int_\gamma d\log f\Bigr) \in \C^\times.
  \ee
  Here $p$ is a point on $\gamma$ and the integrals start from $p$.  The result is independent of the choices of the branch of $\log f$  and the initial point $p$. 
  
It is   important for some applications that the  construction works for the group defined using all holomorphic functions on a punctured Riemann surface,  rather than just 
the meromorphic ones.

 In particular, in the special case when $\Sigma = \C^\times$ and $\gamma$ is a loop around zero, we get a holomorphic variant of the Kac-Moody group extension: 
 \be \la{KMG}
1 \lra \C^\times \lra \widehat \G({\rm Hol}(\C^\times)) \lra \G({\rm Hol}(\C^\times)) \lra 1, 
\ee

\paragraph{3. The determinant line bundles.} Using (\ref{KMG}), we get  an explicit construction of the determinant line bundle on the affine Grassmannian $\widehat \G((t))/ \G({\cal O})$.  
Similarly, we get  an explicit  construction of the determinant line bundle on  ${\rm Bun}_\G$.

\paragraph{4. $K_2-$Lagrangians in  moduli spaces of $\G-$local systems on $S$.}    Recall    the moduli space ${\cal U}_{\G, S}$  of     
  $\G-$local systems on a  punctured surface $S$, with unipotent monodromies around the punctures,  and a reduction to a Borel subgroup at each puncture, called a framing.

Let ${\rm M}$ be a threefold whose boundary   $\overline S$ is obtained by filling the punctures on $S$. Consider 
 the subspace ${\cal M}_{\G, M}\subset {\cal U}_{\G, S}$  parametrising  framed unipotent 
$\G-$local systems on $S$ which   extend to  
 ${\rm M}$.
 
 \bt   i) The moduli space ${\cal U}_{\G, S}$ is $K_2-$symplectic. 
 
ii)  The moduli subspace  ${\cal M}_{\G, M}$ is a $K_2-$Lagrangian subspace of the moduli space ${\cal U}_{\G, S}$. 
 
iii) There is the  {\it motivic volume} map, defined  at the generic point ${\cal M}^\circ_{\G, {\rm M}}$ of ${\cal M}_{\G, {\rm M}}$, with values  in the Bloch group of $\C$:
 \be \la{72}
{\rm Vol}_{\rm mot}:  {\cal M}^\circ_{\G, {\rm M}} \lra \B_2(\C).
  \ee
 \et
 
 \begin{proof} Pick a cocycle ${\rm C}^{(\bullet)}$ representing the class ${\rm c}_2$.
 
 i) Take an ideal triangulation ${\cal T}$ of $\overline S$, i.e. a triangulation with the vertices at the punctures. Take a generic framed $\G-$local system ${\cal L}$ on $S$. 
 Since its monodromy around each puncture  is a regular unipotent element, there exists a unique decorated flags ${\cal F}_p$ near every puncture $p$ which is invariant under the monodromy around $p$. 
 For each triangle $t$ of ${\cal T}$, there is a configuration of three decorated flags $({\cal F}^t_1, {\cal F}^t_2,  {\cal F}^t_3) \in {\rm Conf}_3({\cal A}_\G)$ 
 obtained by restricting the   ${\cal L}$ and the three flat sections of the associated to ${\cal L}$  local system of decorated flags  near each vertex of $t$ to the triangle $t$. Then we have an element 
 \be \la{ELK2}
 W^{\cal T}_S:=\sum_{t \in {\cal T}}{\rm C}^{(2)}({\cal F}^t_1, {\cal F}^t_2,  {\cal F}^t_3) \in K_2(\Q({\cal U}_{\G, S})).
 \ee
 Its tame symbol is zero. Indeed, the tame symbol of each of the terms is a sum of the three standard terms provided by the element ${\rm C}^{(3)}$, 
 associated with the edges of the triangle $t$, but for each edge $E$, the contributions 
 of the two triangles cancel each other. The element $W^{\cal T}_S$ does not depend on the choice of the triangulation ${\cal T}$ since a flip of the triangulation 
 ${\cal T} \to {\cal T}'$ at an edge $E$ amounts to
 \be
 W^{\cal T}_S -   W^{{\cal T}'}_S = \delta {\rm C}^{(1)}({\cal F}^r_1, {\cal F}^r_2,  {\cal F}^r_3, {\cal F}^r_4), ~~~~~~.
  \ee
  where $r$ is the rectangle of the triangulation associated with the edge $E$, and  $({\cal F}^r_1, {\cal F}^r_2,  {\cal F}^r_3, {\cal F}^r_4)\in {\rm Conf}_4({\cal A}_\G)$ is the quadruple of flags associated to the rectangle. \\

 ii) Take a triangulation ${\cal T}_{\rm M}$ of the threefold ${\rm M}$ extending the triangulation ${\cal T}$ of $\overline S$. Then just as  above, one assigns to each tetrahedron ${\rm T}$ of this triangulation 
 a configuration of $4$ decorated flags $({\cal F}^{\rm T}_1, {\cal F}^{\rm T}_2,  {\cal F}^{\rm T}_3, {\cal F}^{\rm T}_4) \in {\rm Conf}_4({\cal A}_\G)$ and apply to it the element ${\rm C}^{(1)}$:
 \be \la{ELK3}
{\rm Vol}^{}_{\rm mot}: = \sum_{{\rm T} \in {\cal T}_{\rm M}}{\rm C}^{(1)}({\cal F}^{\rm T}_1, {\cal F}^{\rm T}_2,  {\cal F}^{\rm T}_3, {\cal F}^{\rm T}_4) \in \B_2(\Q({\cal M}_{\G, {\rm M}})).
 \ee 
 This element is sign-invariant under the action of the group $S_4$ by Proposition \ref{71}, and thus does not depend on does not depend on the choice of the order of the four flags.
 It also does not depend on the triangulation. Indeed, altering a triangulation by a 2 by 3 Pachner move related to the five tetrahedra whose vertices are decorated by the five flags 
 ${\cal F}_1, \ldots,  \widehat {\cal F}_i, \ldots , {\cal F}_5$  amounts to changing element 
 (\ref{ELK3}) by 
 \be\la{zbg}
 \sum^5_{i=1}{\rm C}^{(1)}({\cal F}_1, \ldots,  \widehat {\cal F}_i, \ldots , {\cal F}_5) \in \B_2(\Q({\cal M}_{\G, {\rm M}})).
\ee
 The cocycle property of ${\rm C}^\bullet$ implies that applying the  Bloch complex differential $\delta$  to (\ref{zbg}) we get zero:
  $$
 \delta \sum^5_{i=1}{\rm C}^{(1)}({\cal F}_1, \ldots,  \widehat {\cal F}_i, \ldots , {\cal F}_5) =0.
 $$ 
 Therefore    $(\ref{zbg})=0$   by a $K-$theoretic argument very similar to the one in the proof of Proposition \ref{71}.
 
 Next, denote by   $j: {\cal M}_{\G, {\rm M}} \subset {\cal U}_{\G, S}$ the natural inclusion. Since ${\rm C}^{(\bullet)}$ is a cocycle:
 \be
 \begin{split}
 \delta {\rm Vol}^{}_{\rm mot} =  &\sum_{{\rm T} \in {\cal T}_{\rm M}}\delta {\rm C}^{(1)}({\cal F}^{\rm T}_1, {\cal F}^{\rm T}_2,  {\cal F}^{\rm T}_3, {\cal F}^{\rm T}_4)\\
&=\sum_{{\rm t} \in {\cal T}} {\rm C}^{(2)}({\cal F}^{\rm t}_1, {\cal F}^{\rm t}_2,  {\cal F}^{\rm t}_3)  \\
&\stackrel{{\rm def}}{=}j^* W_S^{\cal T} \in \Lambda^2\Q({\cal M}_{\G, {\rm M}})^*.\\
 \end{split}
 \ee
The second $=$ is  because the contributions of the internal triangles cancel out. The third equality is valid by the definition of   $j^*$. 
 Therefore  $[j^* W_S^{\cal T} ]=0$ in $K_2(\B_2(\Q({\cal M}_{\G, {\rm M}})))$. The  claim ii) is proved.  \\
 
 iii) Specializing the element (\ref{ELK3}) to any generic complex point of $x$ we get the motivic volume map (\ref{72}).   
 Its composition with the map  $\B_2(\C)\lra \R$, provided by the Bloch-Wigner dilogarithm, is a volume map, generalizing the volume of a hyperbolic threefold. 
  \end{proof}
  
 \paragraph{5. A local combinatorial formula for the second Chern class of a $\G-$bundle.} 
Recall the weight two exponential complex of sheaves on a complex manifold $X$ \cite{G15}: 
\be
\Z(2)\lra {\cal O}(1)  \lra \Lambda^2{\cal O} \stackrel{\wedge^2{\rm exp}}{\lra}   \Lambda^2 {\cal O}^*.
  \ee
 Here the second arrow is $2\pi i \otimes f\lms 2\pi i  \wedge f$, and the last one is $f \wedge g \lms {\rm exp}(f) \wedge {\rm exp}(g)$. 
  It is a complex of sheaves in the analytic topology on $X$, exact modulo torsion.

We sheafify the Bloch complex to a complex of sheaves 
$B_2({\cal O}) \lra \Lambda^2 {\cal O}^*$ and define  
  a map of complexes 
 \begin{displaymath}
    \xymatrix{
& \ar[dl]{  \rm R}_2({\cal O}) \ar[d]\ar[r]&{  \Z}[{\cal O}]     \ar[r]^{\delta}   \ar[d]_{{\Bbb L}_2} &  \Lambda^2 {\cal O}^*  \ar[d]^{=}  \\
\Q(2)\ar[r]  &{\cal O}(1)  \ar[r]  &\Lambda^2{\cal O} \ar[r]^{\wedge^2{\rm exp}}&  \Lambda^2 {\cal O}^*    \\}
         \end{displaymath} 
  To define the  map ${\Bbb L}_2$,    recall the dilogarithm function, with the two accompanying logarithms: 
\be \la{4.10.15.1}
{\rm Li}_2(x):= \int_0^x\frac{dt}{1-t}\circ \frac{dt}{t}, ~~~~-\log (1-x)= \int_0^x\frac{dt}{1-t},
~~~~ \log x:= \int_0^x \frac{dt}{t}.
\ee
Here all integrals are along the same path from $0$ to $x$. The last one is regularised  
using the tangential base point  at $0$ dual to $dt$.  
Then  we set, modifying slightly the original construction of   Bloch \cite{B77}-\cite{B78},  
\be \begin{split}
&{\rm L}_2(x) :=  {\rm Li}_2(x) + \frac{1}{2}\cdot  \log (1-x) \log x +\frac{(2\pi i)^2}{24},  \\
&{\Bbb L}_2(\{x\}_2):= \frac{1}{2}\cdot \log(1-x) \wedge \log x + 2\pi i \wedge 
\frac{1}{2\pi i}{\rm L}_2(x).\\
\end{split}
\ee
We keep the summand $\frac{(2\pi i)^2}{24}$ in ${\rm L}_2(x)$, although 
it does not change $2\pi i \wedge 
\frac{1}{2\pi i}{\rm L}_2(x)$  since  $2\pi i\wedge \frac{2\pi i}{24}=0$ in $\Lambda^2\C$.   The key fact is \cite[Lemma 1.6]{G15}
the map ${\Bbb L}_2$ is well defined on $\Z[{\cal O}]$, i.e. does not depend on the monodromy of 
the logarithms and the dilogarithm along the path $\gamma$ in (\ref{4.10.15.1}).  It evidently provides a map of complexes. 
So one has ${\Bbb L}_2: {\rm Ker}~\delta \lra  {\cal O}(1)$. Furthermore, we have  
\be
{\Bbb L}_2({\rm Ker}~\delta) \subset \C(1), ~~~~~~{\Bbb L}_2({\rm R}_2({\cal O})) \subset   \Q(2).
\ee 

Given a $\G-$bundle ${\cal L}$ over a complex manifold $X$, pick an open   by   discs $\U_i$ and choose a section $g_i$ of ${\cal L}$ over $\U_i$. Then we define a 4-cocycle for the Chech cover $\{\U_i\}$ with values in 
$  \Q(2)$ by setting 
\be
\U_{i_1}\cap \ldots \cap \U_{i_5} \lms \sum_{k=1}^5 (-1)^k  {\Bbb L}_2\Bigl({\rm C}^{(1)}(g_{i_1}, \ldots, g_{i_k }, \ldots, g_{i_5 })\Bigr) \in  \Q(2). 
\ee
The main result of this paper implies that it represents the  second Chern class $c_2({\cal L})$. 
 This  is a local combinatorial formula for $c_2({\cal L})$, in the spirit of the  Gabrielov-Gelfand-Losik  combinatorial formula  \cite{GGL} for the first Pontryagin class. 
 See an elaborate discussion of the simplest example  in \cite[Section 1.7]{G15}.\\
 
 We  conclude   that, although given a cocycle ${\rm C}^\bullet$ the above constructions are very transparent, the cocycle itself for $\G \not = {\rm SL_m}$ 
 is rather complicated, 
 and can not be written  without    the  cluster technology. On the other hand, 
 for $\G={\rm SL_n}$ the cocycle is  simple  and   canonical, see \cite{G93}, \cite[Sections 4.2-4.3]{G15}. 

\section{Quantum deformation of the  cohomology group  ${\rm H}^3_{\rm meas}(\G(\C), \R)$.} \la{S9}

The cluster construction of the second motivic Chern class provides at the same time  its quantum deformation. Let us 
  explain a part of the story:  quantum deformation of the measurable  cohomology class 
   $$
 \beta_3 \in {\rm H}^3_{\rm meas}(\G_\C, \R)
 $$
Recall that a $3-$cocycle representing the class $\beta_3$  is a measurable function $\beta_3(g_1, ..., g_4)$ on 
$\G_\C^4$. Our construction of the element ${\rm C}^{(1)}$ gives an explicit formula for this function as a sum  
  of  Bloch-Wigner dilogarithms:
\be \la{pop}
  \beta_3(g_1, ..., g_4) = \sum_j {\cal L}_2(z_j), ~~~~g_i \in \G(\C).
\ee
      Here $z_i$ are certain rational functions on $(\G/\B)^4$.    Let us define a  quantum deformation of this cocycle.\\
      
Recall the cyclically invariant cluster Poisson  structure 
on $\G^4$ \cite{GS19}, providing a q-deformed algebra of regular functions 
${\cal O}_q(\G^4)$. The quantum analog of  cocycle $\beta_3$ lies in its a  formal completion: 
\be \la{BBBB}
{\cal B}_3\in \widehat {\cal O}_q(\G^4).
\ee
The element ${\cal B}_3$   satisfies the following analog of the 3-cocycle relation. Recall the natural 
 projections 
 $$
 s_i:\G^5 \lra \G^4, ~~~~\mbox{where} ~~ i \in \Z/5\Z,
 $$
They are  cluster Poisson maps. Therefore they give rise to a map of algebras
 $$
 s_i^*: \widehat {\cal O}_q(\G^4) \lra \widehat {\cal O}_q(\G^5).
  $$
 \bt
 There is an element ${\cal B}_3$ in (\ref{BBBB}) which satisfies the following  quantum cocycle rcondition: 
 \be \la{CALLB}
 \prod_{j=1}^5 s_{2j+1}^*{\cal B}_3=1. 
\ee
\et

\begin{proof} The space $(\G/\B)^4$ carries a cyclically invariant cluster Poisson structure.  Therefore we get the algebra ${\cal O}_q((\G/\B)^4)$. The projection
$
\pi: \G^4 \lra (\G/\B)^4
$ 
is a cluster Poisson map, providing the  map
$$
\pi^*: {\cal O}_q((\G/\B)^4) \lra {\cal O}_q(\G^4).
$$
Let us define an element 
$$
{\bf B}_3 \in {\cal O}_q((\G/\B)^4)
$$
which satisfies the analog of the cocycle condition  
\be \la{wewe}
\prod_{j=1}^5 s_{2j+1}^*{\bf B}_3=1.
\ee
 Then we set
$$
{\cal B}_3:= \pi^*{\bf B}_3. 
$$
It automatically satisfies the cocycle condition. Let us implement this program.\\

 Recall   the quantum dilogarithm power series, convergent if $|q|<1$ for any $Z\in \C$:
 $$
 \Psi_q(Z) = \frac{1}{(1+qZ) (1+q^3Z) (1+q^{5}Z) ...}
 $$
  The element ${\bf B}_3$  is defined as a certain product of the quantum dilogarithm power series 
\be \la{SEQC}
 {\bf B}_3 = \prod_j \Psi_q({Z}_j), ~~~~Z_j \in {\cal O}_q((\G/\B)^4).
 \ee
  The functions $\{Z_j\}$ are   $q-$deformations of  functions $z_j$ in (\ref{pop}),   $\lim_{q\to 1}Z_j = z_j$, defined as follows. \\
 
  Pick a reduced decomposition ${\bf i}$ of the longest element $w_0$.  
   Consider a quadrilateral $Q$ with a special side $F$,   and a diagonal $E$. There are 
  two marked angles   of the quadrilateral: the one in the triangle with the base $F$, opposite to the base, 
  and the one in the second triangle, opposite to the diagonal $E$. Each of the two triangles is counterclockwise oriented. Therefore the reduced decomposition ${\bf i}$ provides a cluster Poisson coordinate system on the moduli space 
  ${\cal P}_{\G, 3}$ assigned to each triangle, and hence, by the amalgamation, a cluster Poisson structure on ${\cal P}_{\G, 4}$  
  and thus on ${\rm Conf}_4({\cal B})$. 
  
Take a pentagon ${\rm P}_5$ decorated by the  flags $\B_1, ..., \B_5$, providing a point of the moduli space 
  ${\rm Conf}_5({\cal B})$. Take a triangulation ${\rm T}_1$ of the pentagon. Pick one of the diagonals and denote it by $\F$. 
  It determines a quadrilateral $Q_F$ inscribed into the pentagon, which has $F$ as 
  its side, and the triangle $t_E$  outside of $Q_F$, with the base $E$.  Mark the angle of the triangle $t_E$ opposite to the side $E$. 
  Then each of the three triangles of the pentagon has a marked angle. Therefore the reduced decomposition ${\bf i}$ provides a cluster Poisson coordinate system on the  space  ${\rm Conf}_5({\cal B})$. Now  flip the triangulation given by the edges $(E,F)$ 
  at the edge $F$.   Relabel the new pair of  edges $(E, E_1)$   as $(F_1, E_1)$, setting $F_1:=E, E_1:=E$. 
  Then we can assign to the new triangulation $(E_1, F_1)$ a similar cluster Poisson coordinate system on ${\rm Conf}_5({\cal B})$. 
 \vskip 2mm

 {   The key point is that the flip of   triangulation at the edge $F$ alters cluster Poisson coordinates only in  the quadrileteral $Q_F$, that is determined by the four flags at its vertices. It is realized as an ordered sequence of mutations, provided by    the cluster Poisson rational functions $Z_1, ..., Z_N$ 
 on  ${\rm Conf}_4({\cal B})$. Each mutation is given by the conjugation by $\Psi_q(Z_i)$. We use the   sequence 
 $\{Z_j\}$ to define the element ${\bf B}_3$ in (\ref{SEQC}). The elements $\{z_j\}$ in (\ref{pop}) are defined as the $q=1$ specialization of the elements $\{Z_j\}$. } 
 \vskip 2mm
  
   The main difference between the classical and quantum cocycles $\beta_3$ and ${\bf B}_3$ is that the elements $\{Z_i\}$ do not commute, and so  their order is an essential  part of the definition of the element ${\bf B}_3$. 
   \vskip 2mm
   
   Traditionally each mutation is given by the conjugation by $\Psi_q(Z)$ followed by a monomial transformation, and a cluster Poisson transformation is defined as a composition of such elementary transformations. However one can also 
 define a {\it reduced mutation} as just  the conjugation by $\Psi_q(Z)$, and define the {\it reduced cluster Poisson transformation} as 
 the composition of reduced mutations  \cite[Proposition 2.4]{GS16}. 
  \vskip 2mm

Performing this procedure five times, we get the original triangulation $(E,F)$, as well as the original cluster Poisson coordinate system. The sequence of cluster Poisson coordinates given  by the sequence of mutations realizing the flip of the diagonal $F_i$ on the step $i$ is denoted by 
 $Z^{(i)}_1, ..., Z^{(i)}_N$.  Then the ordered sequence of cluster Poisson coordinates we need is given by the $5N$ functions 
 \be
 Z^{(1)}_1, ..., Z^{(1)}_N;  ~Z^{(1)}_1, ..., Z^{(1)}_N; ~\ldots ; ~Z^{(5)}_1, ..., Z^{(5)}_N.
  \ee
  \bp
  The following product is equal to $1$:
\be \la{SEQC1}
 \prod_j \Psi_q({Z}^{(5)}_j) \cdot  \prod_j \Psi_q({Z}^{(4)}_j)  \cdot \ldots \cdot  \prod_j \Psi_q({Z}^{(1)}_j) =1.
\ee 
\ep

 \begin{proof} If a reduced cluster transformation 
 is the identity map, then the product of the corresponding $\Psi_q(Z_i)$ in the completed q-deformed algebra is equal to $1$ 
 \cite{K12}, \cite{K13}, cf \cite[Theorem 3.2]{GS16}.   \end{proof}
 
 The relation (\ref{SEQC1}) is  equivalent to   relation (\ref{wewe}) on   elements (\ref{SEQC}), and hence to relation (\ref{CALLB}). 
  \end{proof}

To justify the name quantum dilogarithm for the formal power series $\Psi_q(z)$, recall   the 
 dilogarithm:
$$
{\rm L}_2(x):= \int_0^x\log(1+t)\frac{dt}{t} = -{\rm Li}_2(-x).
$$
It has a $q$-deformation, called the {\it $q$-dilogarithm power series}:
$$
{\rm L}_2(x;q):= \sum_{n=1}^{\infty}\frac{x^n}{n(q^{n}-q^{-n})}.
$$
One has the identity
$$
\log {\bf \Psi}_q(x)  = {\rm L}_2(x;q).
$$

If $|q|<1$ the power series ${\bf \Psi}_q(x)$ converge, providing 
an analytic function in $x \in \C$. If in addition to this $|x|<1$, the 
$q$-dilogarithm power series also converge. There are asymptotic expansions when $q\to 1^-$:
\begin{equation} \label{sine}
{\rm L}_2(x;q) \sim \frac{{\rm L}_2(x)}{\log q^2}, \qquad 
{\bf \Psi}_q(x) \sim {\rm exp}\Bigl(\frac{{\rm L}_2(x)}{\log q^2}\Bigr). 
\end{equation}
 Using this one can show that  the quantum cocycle relation (\ref{SEQC1}) implies  the classical one if  $q\to 1$.   
   \vskip 2mm   
   
   In the case when $\G = {\rm PGL}_2$, the element ${\bf B}_3$ is just the quantum dilogarithm $\Psi_3(Z)$, and our cocycle relation 
    reduces to the Faddeev-Kashaev \cite{FK} pentagon relation for the quantum dilogarithm. \\
    
    Note also that there is a version of the quantum cocycle where the role of the power series $\Psi_q(Z)$ is played by the quantum modular 
    dilogarithm $\Phi_\hbar(z)$. The main difference is that now the cocycle is well defined for any $q \in \C$, and 
    is understood as an operator acting in a Hilbert space.

\section{Cluster structures and motivic cohomology: conclusion}

\paragraph{1. Conclusions.} 1. Formula (\ref{KF}) tells that the cocycle ${\rm C}^{(3)}$ is just the Killing form (\ref{KF1}), written as a bilinear form (\ref{KF}),   
translated isomorphically into  the middle group  in 
 (\ref{MGR}),   thus interpreted   as  
a cocycle for  
$\rH^1(\G, \underline K_2)$.  
To make  the bilinear form (\ref{KF}) from  the quadratic form (\ref{KF1}) we need the coefficients $\frac{1}{2}$ in front of $d_k{\rm C}_{kj}$. 
Indeed, the left and the right factors in the bilinear expression (\ref{KF}) have entirely different meanings in  (\ref{MGR}) as, respectively, Bruhat divisors and functions on them. 
A posteriori this explains why the exchange matrix $\varepsilon_{ij}$ has half integral values between the frozen variables. \\

2. The cluster structure of the elementary variety ${\cal A}(k)$,   $k\in {\rm I}$, is  determined by the  following  facts:

i) The corresponding element 
 $W(k)$ is decomposed into a sum of two terms  
 $$
 W(k) = W'(k) + W_\Delta(k),
 $$
  where $W_\Delta(k)$ is the pull back $\tau_k^*$ 
 of the element $W$ from the space ${\rm Conf}_3({\cal A}_{\rm SL_2}) $ for the canonical projection
 \be
 \tau_k: {\cal A}(k) \lra {\rm Conf}_3({\cal A}_{\rm SL_2}).
 \ee
 
 ii)   The residue of  $W(k)$ at the "right side of the quiver"  is given by the cocycle ${\rm C}^{(3)}$. 
 
 Equivalently, 
the cocycle ${\rm C}^{(3)}$ is the residue of   $W_{\bf c}$ at the right side of the triangle $t$.     

Indeed,   the tame symbol calculation (\ref{WC}) nails the shape of the   quiver ${\rm J}(k)$ of   ${\cal A}(k)$. 
Namely, the exchange matrix for the right   side of the quiver ${\rm J}(k)$ is the negative of the one for the left edge, as the argument 
in the end of the proof of Theorem \ref{BTH} shows. It is determined by the   cocycle ${\rm C}_3$, and the latter is fixed by the Killing form, as   discussed above.    \\

3. The element $W_{\bf c}$ on ${\rm Conf}_3({\cal A})$  determines the cluster structure on this space. 
The element $W_{\bf c}$ is forced onto us as 
 the one whose tame symbol is given by formula (\ref{WC}).  Therefore its  existence   follows from   $\rH^4(\B\G, \Z_{\cal M}(2))=\Z$. 

 Although such an element $W_{\bf c}$ is not unique, the difference between any 
  two  $W_{\bf c}$  and 
 $W'_{\bf c}$  of them   is a cocycle, providing a class  $[W_{\bf c} -W'_{\bf c} ] \in \rH^0(\G^2, \underline K_2)/K_2(\Z)$.  Note that $K_2(\Z) = \Z/2\Z$. On the other hand,  
\be \la{NOKT}
 \rH^0(\G^2, \underline K_2)/K_2(\Z)=0.
  \ee
  This  implies the crucial, and one of the most challenging, properties of the element $W_{\bf c}$:     its class 
     in the group $K_2$ of the field of functions $\Q({\rm Conf}_3({\cal A}))$ is {\it twisted cyclically invariant} \cite[Section 7]{GS19}. Indeed, it  follows from    (\ref{NOKT}), 
   since the tame symbol of $W_{\bf c}$, given by (\ref{WC}),  is  twisted cyclic shift invariant   on the nose. \\
   
 4.   One also has 
 \be \la{NOKT1}
 \rH^0(\G^3, \underline K_2)/K_2(\Z)=0. 
 \ee
 This makes evident another crucial fact, this time  about the cluster structure of the space ${\rm Conf}_4({\cal A})$: the {\it flip invariance} of the   $K_2-$class of the element $W$ 
 on ${\rm Conf}_4({\cal A})$, 
 see  
  paragraph 4   in Section \ref{SECT4}. 
   Indeed, the vanishing (\ref{NOKT1}) implies that this $K_2-$class its determined by its tame symbol. The latter, as follows   from 
    (\ref{WC}),  is  the sum of the contributions of the four sides of the rectangle, and thus evidently flip invariant.    \\
 
 5.  The cluster structure of the   moduli space  ${\cal A}_{\G, \bS} $ is constructed by starting from the cluster structure of the space 
 ${\rm Conf}_3({\cal A})$. Next, using its twisted cyclic invariance, we introduce  the cluster structure on ${\rm Conf}_4({\cal A})$ via the amalgamation. The flip invariance  of the latter 
 allows to extend the construction of the cluster structure  by the amalgamation to the whole surface, and guarantees its $\Gamma_\bS-$equivariance. 
  The cluster Poisson structure of the  space  ${\cal P}_{\G, \bS} $ is deduced from this. 
  Therefore the discussion above explains, for the first time,  why the cluster structure on the dual pair of moduli spaces  $({\cal A}_{\G, \bS}, {\cal P}_{\G, \bS})$   should exist.  \\  
   
 6. The fact that the number of functions entering $W_{\bf c}$ is the same as the dimension   of ${\rm Conf}_3({\cal A})$ 
   is irrelevant for the motivic considerations described in this paper, although the collection of different clusters was used essentially to prove relation (\ref{WC}).

 However  what is needed for many applications, e.g. for the cluster quantization,  is not just the fact that the $K_2-$class $[W_{\bf c}]$   is twisted cyclically invariant, but that 
   the equivalence between different elements $W_{\bf c}$ is achieved by cluster transformations. 
   This, and the amazing fact that the number of functions entering $W_{\bf c}$ is equal to    ${\rm dim}{\rm Conf}_3({\cal A})$,   
   shows that the construction of the second motivic Chern class  capture    many, but not all, cluster features of the space ${\rm Conf}_3({\cal A})$.

\paragraph{2. Generalizations.} The truncated cocycle  $ ({\rm C}^{(2)},  {\rm C}^{(3)}) $ gives  the second Chern class in the $K_2-$cohomology:
  \be \la{C2s}
 {\rm c}^{\rm M}_2 \in \rH^2(\B\G_\bullet, \underline K_2).
 \ee 
 
 For $\G = {\rm SL}_m$,  there is   an explicit construction  of all Chern classes   in the Milnor $K-$theory \cite{G93}:
  \be \la{C2ss}
 {\rm c}^{\rm M}_m \in \rH^m(\B{\rm GL}_\bullet, \underline K^{\rm M}_m).  
 \ee
Its analogs   for other groups $\G$ is not known for $m>2$.  Note that these are the classes  \
\be \la{C2sss}
 {\rm c}^{\rm M}_{d_m} \in \rH^{d_m}(\B{\rm G}_\bullet, \underline K^{\rm M}_{d_m}).\ \ \ \ m \in \{1, ... , {\rm rk}(\G)\}.
 \ee 
  where $\{d_m\}$ are the exponents of $\G$. 
It would be very interesting to find them.   An interesting question   is whether we would need a more general notion than 
 the cluster structure to do this. 
 
Furthermore, there is an explicit construction of the third motivic Chern class, see \cite{G15}:
\be
{\rm c}_3\in \rH^6(\B{\rm GL}_{m\bullet}, \underline \Z_{\cal M}(3)).
\ee
This class is crucial to understand the Beilinson regulator for the weight $3$. However, strangely enough, the class ${\rm c}_3$ did not appear yet in any geometric/Physics applications like the ones   
 in Subsection 1.11. 

 It would be   interesting to construct  explicitly  the third motivic Chern class for any classical group $\G$. Note that although $d_1=2$, for the classical $\G$ we have $d_2=3$, while otherwise $d_2>3$.


\begin{thebibliography}{BGSV}

\bibitem[Be]{Be} Beilinson A.A.: {\it Higher regulators and special values of L-functions}. 1984. 

\bibitem[B77]{B77} Bloch S.: {\it Higher regulators, Algebraic K-theory, and Zeta functions of elliptic curves.} 
Irvine lecture notes. CRM Monograph series, AMS 2000. Original  preprint of 1977.  

\bibitem[B78]{B78} Bloch S.: {\it Applications of the dilogarithm function in algebraic K-theory and algebraic geometry}, in: Proc. Int. Symp. on Alg. Geometry, Kinokuniya, Tokyo 1978. 

\bibitem[BD]{BD} Brylinsky J.L., Deligne P.: {\it Central extensions of reductive groups by ${\bf K}_2$}. Publ. Math IHES,  94 (2001), pp. 5-85. 

\bibitem[D]{D}  Deligne P.:  {\it Le symbole mod\'er\'e}, Publ. Math. IHES 73 (1991), 147 - 181. 

\bibitem[DGG]{DGG} Dimofte T., Gabella M., Goncharov A.B.: {\it K-Decompositions and 3d Gauge Theories}   \href{https://arxiv.org/abs/1301.0192}{arXive 1301.0192.}     

\bibitem[EKLV]{EKLV} Esnault E., Kahn B., Levine M., Viehweg E. H.: {\it The Arason invariant and mod 2 algebraic cycles}, JAMS 11
(1998), no. 1, 73-118.

\bibitem[FK]{FK}Faddeev L.D., Kashaev R.:  {\it Quantum dilogarithm}. Mod.Phys.Lett. A 9 427 (1994). 
\href{https://arxiv.org/abs/9310070.}{arXive 9310070.}     
\bibitem[FG03a]{FG1} Fock V.V., Goncharov A.B. 
{\it Moduli spaces of local systems
    and higher Teichmuller theory}. Publ. Math. 
IHES, n. 103 (2006) 1-212.  \href{https://arxiv.org/abs/AG/0311149}{arXive AG/0311149.}   
\bibitem[FG03b]{FG03b} Fock V.V., Goncharov A.B.:
 {\it Cluster ensembles, quantization and the dilogarithm}. 
Ann. Sci. \'Ecole Norm. Sup. vol 42, (2009) 865-929. 
 \href{https://arxiv.org/abs/AG/0311245}{arXive AG/0311245.}   
\bibitem[FG05]{FG05} Fock V.V., Goncharov A.B.: {\it Cluster ${\cal X}-$varieties, amalgamation, and Poisson - Lie groups.}
 Algebraic geometry and number theory, 2006.  \href{https://arxiv.org/abs/math/0508408}{arXiv:math/0508408}. 
  \bibitem[FZI]{FZI} Fomin S., Zelevinsky A.: 
{\it Cluster algebras. I}. 
J. Amer. Math. Soc. 15 (2002), no. 2, 497--529.
\bibitem[GGL]{GGL} Gabrielov A., Gelfand I.M., Losik M.: 
{\it Combinatorial computation of characteristic classes} I, II,
Funct. Anal. i ego pril. 9 (1975) 12-28, ibid 9 (1975) 5-26. 
\bibitem[GZ]{GZ} Garoufalidis S., Zickert C.: {\it The symplectic properties of the $PGL(n,C)$-gluing equations,}
 \href{https://arxiv.org/abs/1310.2497.}{arXive 1310.2497}.   
\bibitem[G91]{G91} Goncharov A.B. {\it Polylogarithms and motivic Galois groups.} Proceedings of Symposia in Pure Mathematics, 
 Motives (Seattle, WA, 1991) 55, 43-96.
\bibitem[G93]{G93} Goncharov, A. B. {\it Explicit construction of characteristic classes}. 
{Advances in Soviet Mathematics,} 16, v.\ 1, 
    Special volume dedicated to 
    I.M.Gelfand's 80th birthday, 169 - 210, 1993.
    \bibitem[G15]{G15} Goncharov A.B.: {\it Exponential complexes, period morphisms, 
and characteristic classes.}  Ann. de la Faculte des Sci. de Toulouse, Vol XXV, n 2-3, (2016) pp 397-459. 
  \href{https://arxiv.org/abs/1510:07270}{arXive 1510:07270}.   
  \bibitem[GS16]{GS16} Goncharov A.B., Shen L.:  {\it Donaldson-Thomas trasnsformations of moduli spaces of G-local systems.}  
  Advances in Mathematics 327, 225-348.  \href{https://arxiv.org/abs/1602.06479}{arXive 1602.06479}.  
      \bibitem[GS19]{GS19} Goncharov A.B., Shen L.:  {\it Quantum geometry of moduli spaces of local systems and Representation Theory}.   
  \href{https://arxiv.org/abs/1904.10491}{arXive 1904.10491},   version 2.

\bibitem[K12]{K12} Keller B.: {\it Cluster algebras and derived categories}. Derived categories in algebraic geometry, 123-183, EMS Ser. Congr. Rep., Eur. Math. Soc., Z\"urich, 2012.   \href{https://arxiv.org/abs/1202.4161}{arXive 1202.4161}. 

\bibitem[K13]{K13} Keller B.:  {\it Quiver mutation and combinatorial DT-invariants}. Contribution to the FPSAC
2013. http://webusers.imj-prg.fr/ bernhard.keller/publ/index.html.

 \bibitem[LS]{LS} Laszlo, Y., Sorger Ch.: 
{\it The line bundles on the moduli of parabolic $\G-$bundles
over curves and their sections}. 
Annales scientifiques de l'ENS. 4e
s\'erie, tome 30, no 4 (1997), p. 499-525.
 \bibitem[Le]{IL}
I.~Le:
\newblock {\it Cluster structures on higher {T}eichm{\"u}ller spaces for classical
  groups.}
 \href{https://arxiv.org/abs/1603.03523}{ arXiv:1603.03523}.
 
  {\it An Approach to Cluster Structures on Moduli of Local Systems for General Groups.} 
\newblock \href{https://arxiv.org/abs/1606.00961}{arXiv:1606.00961}.

 \bibitem[LS]{LS}    Lee R., Szczarba R.: {\it The group $K_3(\Z)$ is a cyclic group of order $48$.} Annals of Mathematics, 104, (1976), 31-60.

 \bibitem[S]{S} Suslin A.A.: {\it The group $K_3$ of a field and the Bloch group} Proc. Steklov Inst. Math., 183 (1991), 217 - 239. 
 \end{thebibliography}
\end{document}